\documentclass[a4paper,11pt]{amsart}
\usepackage{amsmath,amssymb,amsfonts,amsthm,exscale,calc}
\usepackage[latin1]{inputenc}
\usepackage{latexsym,textcomp}
\usepackage{graphicx,epsf}
\usepackage{dsfont}
\usepackage[german, english]{babel}
\usepackage{hyperref}
\usepackage{a4wide}

\theoremstyle{plain}
\newtheorem{lemma}{Lemma}[section]
\newtheorem{prop}[lemma]{Proposition}
\newtheorem{coro}[lemma]{Corollary}
\newtheorem{theorem}[lemma]{Theorem}

\theoremstyle{definition}
\newtheorem{definition}[lemma]{Definition}
\newtheorem{remark}[lemma]{Remark}

\newcommand{\ts}{\hspace{0.5pt}}

\newcommand{\RR}{\mathbb{R}\ts}

\newcommand{\NN}{\mathbb{N}}
\newcommand{\R}{\mathbb{R}\ts}

\newcommand{\LD}{\mathcal D}

\newcommand{\CX}{\mathcal X}

\newcommand{\aV}[1]{\left\Vert #1\right\Vert}
\newcommand{\as}[1]{\langle #1\rangle}
\newcommand{\ov}[1]{\overline{#1}}
\newcommand{\ow}[1]{\widetilde{#1}}
\newcommand{\oh}[1]{\widetilde{#1}_{l,b}}

\newcommand{\eps}{\varepsilon}

\newcommand{\supp}{\mathrm{supp}}

\newcommand{\muc}{\mu^{(c)}}

\newcommand{\Qs}{Q^\#}
\newcommand{\Ls}{L^\#}
\newcommand{\Ds}{\mathcal{D}^\#}

\newcommand{\Hmm}[1]{\leavevmode{\marginpar{\tiny%
$\hbox to 0mm{\hspace*{-0.5mm}$\leftarrow$\hss}%
\vcenter{\vrule depth 0.1mm height 0.1mm width \the\marginparwidth}%
\hbox to 0mm{\hss$\rightarrow$\hspace*{-0.5mm}}$\\\relax\raggedright #1}}}

\newcommand{\cL}{\mathcal{L}}

\begin{document}

\title[Global properties in terms of Green's formula]{Global properties of
Dirichlet forms in terms of Green's formula}

\author[]{Sebastian Haeseler$^1$}
\author[]{Matthias Keller$^2$}
\author[]{Daniel Lenz$^3$}
\author[]{Jun Masamune$^4$}
\author[]{Marcel Schmidt$^5$}

\address{$^1$, $^2$, $^3$, $^5$ Mathematisches Institut, Friedrich Schiller Universit\"at Jena,
  07743 Jena, Germany, http://www.analysis-lenz.uni-jena.de/}
\email{sebastian.haeseler@uni-jena.de}
\email{m.keller@uni-jena.de}
\email{daniel.lenz@uni-jena.de}
\email{marcel.schmidt@uni-jena.de}

\address{$^4$ Division of Mathematics / Research Center for Pure and Applied Mathematics,
Graduate School of Information Sciences, Tohoku University,
6-3-09 Aramaki-Aza-Aoba, Aoba-ku, Sendai 980-8579, Japan, http://www.math.is.tohoku.ac.jp/~jmasamune/}
\email{jmasamune@m.tohoku.ac.jp}

%
% -------------------------------------------------------------%
% -------------------------------------------------------------%
\begin{abstract}
We study global properties of Dirichlet forms such as uniqueness of
the Dirichlet extension, stochastic completeness and recurrence. We
characterize these properties by means of vanishing of a boundary
term in Green's formula for functions from suitable function spaces
and  suitable operators arising from  extensions of the underlying
form. We first present results in   the  framework of general
Dirichlet forms on $\sigma$-finite measure spaces. For regular
Dirichlet forms our results can be strengthened as   all operators
from the previous considerations turn out to be restrictions of a
single operator. Finally, the results are applied to graphs,
weighted manifolds,
 and metric graphs, where the operators under investigation can
be determined rather explicitly.

\end{abstract}
% -------------------------------------------------------------%
\date{\today} %
\maketitle

\tableofcontents

\section*{Introduction}
We study global properties of  Dirichlet forms in terms vanishing or
non-vanishing of a  ``boundary term'' in a  Green's formula. In this
introduction we survey the structure and the results of this paper
without going into  detail. The precise definitions and statements
can then be found in the upcoming sections.

%The theory of Dirichlet forms provides a unified  framework for
%Markov processes such as Brownian motion on Riemannian manifolds and
%jump processes on discrete  graphs.

 The global properties under investigation of a Dirichlet
form on
 a $\sigma$-finite measure space $(X,m)$ are the
following:
\begin{itemize}
  \item Uniqueness of Dirichlet extensions.
  \item Stochastic completeness.
  \item Recurrence.
\end{itemize}

Here, we present a rather complete study of  these properties via
vanishing of  Green's formula
\begin{align}\label{e:GF}\tag{GF}
    \int_{X} \Delta u \; dm =0,
\end{align}
where the function $u$ is in a suitable class of functions and
$\Delta$ is a suitable extension of an operator arising from a
Dirichlet form depending on the problem.

The  study of  these issues by means of  Green's formula is inspired
by very recent work of Grigor'yan/Masamune \cite{GM} on weighted
manifolds. However, while their study is based on methods from
geometry  our approach is rather different. Indeed, we used
techniques from functional analysis  and the full advantage of the
modern abstract Dirichlet form theory. As a consequence, we will
obtain some new characterizations even for weighted manifolds. Our
main strategy in dealing  with the lack of geometric structure in
the general case is to consider extensions of the Dirichlet form
under investigation to suitable spaces  and to associate generators,
which are carefully chosen to optimize the results according to the
problems.

We study these  questions on three levels of abstraction while the
statements become the more explicit the more specific we get. The
most abstract level are general Dirichlet forms on $\sigma$-finite
measure spaces. Secondly, we focus on regular Dirichlet forms and,
finally, we present the results in the concrete contexts of graphs,
weighted manifolds,  and metric graphs.

Let us go into a bit  more detail. We start with the most abstract
level and discuss the three properties listed above.

In Section~\ref{section;Extensions} we  determine whether a Dirichlet extension $Q^{\#}$ of a Dirichlet form $Q$ is different
from $Q$  by existence of a positive subharmonic, but
non-harmonic, $L^{1}$ function $u$. In this case $u$ is
such that \eqref{e:GF} fails and $\Delta$ can be understood as the
``Gaffney Laplacian'' with respect to $Q$ and $Q^{\#}$ and it will
be denoted by $L'$. The main result (Theorem \ref{thm;2.2}) is the
characterization of the uniqueness of the Dirichlet extensions under
the presence of a maximum principle (see Definition \ref{def;MP}).
The maximum principle will be explored further in Section
\ref{section;MP} and a criterion for the maximum principle is given
(Theorem \ref{max principle}). The criterion is quite general, e.g., it
can be applied to various extensions (Silverstein extensions) of regular Dirichlet forms.

Secondly, we consider stochastic completeness  in Section~\ref{section;SC}. A
 Dirichlet form is called stochastically complete if
$(L+1)^{-1}1=1$,  see Definition~\ref{def:sc}  and the discussion
below for  background. We obtain characterizations for stochastic
completeness in terms of \eqref{e:GF} with $\Delta$ being the
$L^2$-generator $L$ or the $L^1$-generator $L^{(1)}$ and we obtain
another characterization by the dual of $L^{(1)}$
 on $L^{\infty}$ (Theorem~\ref{thm;SC}). The statement for the $L^{2}$-generator
is an extension of the corresponding result in \cite{GM} to general Dirichlet forms while the other characterizations
are new even in the case of manifolds.

Finally, in Section~\ref{section;recurrence}, we study recurrence. A
Dirichlet form is called recurrent if $\int_{0}^{\infty}e^{-tL}fdt$
is equal to $0$ or $\infty$ almost everywhere for all non-negative
$L^{1}$ functions $f$, see Definition~\ref{def:rec/trans}.  We
characterize recurrence by \eqref{e:GF} with $\Delta $ being an
extension of $L$ related to the extended Dirichlet space which will
be denoted by $L_{e}$.

We remark that there is a relationship among these global
properties. In general, recurrence implies stochastic completeness.
Furthermore, stochastic completeness implies $Q = \Qs$ when both
forms satisfy a maximum principle. These implications are well-known
and easily follow from our considerations.

After having established the theory in the general setting we zoom
in to the case of regular Dirichlet forms in
Section~\ref{section;Regular}. Under this situation, we are able to
extend all the above generators to one  operator denoted by
$\mathcal{L}$, and apply it to improve the results obtained in the
previous sections. This application is preceded by a version of
Fatou\rq{}s lemma for the reflected Dirichlet form which allows for
the definition of the operator $\mathcal{L}$. Then, in
Subsection~\ref{subsection;Silverstein}, we give two independent
criteria
 (Theorems \ref{H neq H0} and \ref{thm;US1})
for the uniqueness of Dirichlet extensions. The first one is stated
in terms of the existence of positive subharmonic functions, while the second one is phrased via
the validity of a Green formula. The last two subsections,
Subsections \ref{subsection;SC} and \ref{subsection;recurrence}, are
devoted to the study of stochastic completeness and recurrence,
respectively, using $\mathcal{L}$.

In the last the sections, Sections \ref{section;Graph},
\ref{section;Manifold}, and \ref{section;MG}, we  apply the abstract
results obtained of the previous sections to more concrete Dirichlet
forms.

Specifically, in Section \ref{section;Graph}, we  study  graphs
which have the prominent feature  that all of operators are
restrictions of a formal operator $\widetilde L$. This will allow us
to understand these problems in a unified way.  We  emphasize that
we do not need to assume local finiteness in any of our results.
Part of the application to graphs is based on \cite{Sch}.

In Section \ref{section;Manifold}, we study a general Dirichlet form
on a weighted manifold. The main result here is the determination of
the reflected Dirichlet space which generalizes a recent result of
\cite{CF} to the manifold setting. In particular, this result
implies that in this situation all of the introduced operators are
restrictions of some weighted version of the  Laplace Beltrami
operator.

Finally, in Section \ref{section;MG}, a general Dirichlet form on a
metric graph is studied. Part of this application is based on \cite{Hae}.

\textbf{List of some  relevant notation.} As we have explained
above, we will use various different operators associated to the
Dirichlet form and its extensions. Below, we will list them  to
serve as an index (all of them except $\mathcal{L}$ are defined for
a general Dirichlet form):
\begin{itemize}
\item $L: \mathcal D(L) \to L^2(X,m)$ -- the $L^2$-generator of the ``minimal''
Dirichlet form
$Q$ (Section~\ref{section;Extensions}).
\item $\Ls:\mathcal D(\Ls) \to L^2(X,m)$ -- the $L^2$-generator of the
``maximal'' Dirichlet form
$\Qs$ (Section~\ref{section;Extensions}).
\item $L' : \mathcal D(L') \to L^2(X,m)$ -- an extension of both $L$ and $\Ls$
which will be used to
characterize the agreement of $L$ and $\Ls$ (see Proposition~\ref{generator}).
\item $L^{(1)}$ -- the $L^1 (X,m)$-generator
which will be used to characterize stochastic completeness (Section~\ref{section;SC}).
\item $L_e: D(Q)_e \to L^2(X,m)$ -- an extension of $L$
which will be used to characterize recurrence (Section~\ref{section;recurrence}).
\item $\mathcal{L}: D(\mathcal{L}) \to L_{\rm loc}^1 (X,m)$ -- an extension of
all of the above operators,
which will be defined for a regular Dirichlet form (Section~\ref{section;Regular}).
\end{itemize}

\bigskip

\textbf{Acknowledgments.} M.K. and D.L. gratefully acknowledge
partial support from German Research Foundation (DFG). M.S. has been
financially supported by the Graduiertenkolleg 1523/2 : Quantum and
gravitational fields and by the European Science Foundation (ESF) within the project Random Geometry
of Large Interacting Systems and Statistical Physics.

%%%%%%%%%%%%%%%%%%%%%%%%%%%%%%%%%%%%%%%
%%%%%%%%%%%%%%%%%%%%%%%%%%%%%%%%%%%%%%%
\section{The set up}

As mentioned in the introduction this paper is concerned with
global properties of symmetric Dirichlet forms. In this section we
fix some notation and briefly recall the relevant notions and
objects. For further discussions and proofs we refer the reader to
the textbooks \cite{FOT} and \cite{CF}.

Throughout,  we let $m$ be a $\sigma$-finite measure on a measurable
space $X$. For any number $1\leq p\leq\infty$, we denote by
$L^p(X,m)$ the corresponding real valued Lebesgue space with norm
$\|\cdot\|_p$. The scalar product on $L^2(X,m)$  is denoted  by
$\as{\cdot,\cdot}$. The space of all measurable $m$-a.e. defined
real valued functions is denoted by $L^0(X,m)$. For $f,g\in L^0
(X,m)$  we will write $f\wedge g $ to denote the minimum of $f$ and
$g$ and $f \vee g$ to denote the maximum of $f$ and $g$.

A function $C: \RR \to \RR$ is a {\em normal contraction} if $C(0) =
0$ and $|C(x) - C(y)| \leq |x-y|$ holds for all $x,y \in  \RR.$  A
densely defined, non-negative,  closed symmetric bilinear form
 $$Q: D(Q) \times D(Q) \to \RR$$
on $L^2(X,m)$ is called a  {\em Dirichlet form} if $Q$ satisfies the
{\em contraction property}, that is for each normal contraction $C$
and each $u \in D(Q)$ we have $C \circ u \in D(Q)$ and the
inequality
$$Q(C\circ u,C\circ u) \leq Q(u,u)$$
holds. We will write $Q(u) := Q(u,u)$ for $u \in D(Q)$ and set $Q(u)
:= \infty$ if $u \not  \in D(Q)$. The {\em form norm} is given by
$\|u\|_Q:= (Q(u) + \|u\|_2^2)^{1/2}$ and the corresponding {\em form
inner product} will be denoted by $\as{\cdot,\cdot}_Q$.

A Dirichlet form gives rise to a positive self-adjoint operator $L$
on $L^2(X,m)$. This operator is uniquely determined by the equality
$$Q(u,v) = \as{Lu,v},$$
for each $u \in D(L)$ and each $v \in D(Q)$. Here $D(L)$ is the
domain of $L$.  Each such operator coming from a Dirichlet form
yields a strongly continuous  resolvent $(G_\alpha)_{\alpha >0}$ and
a strongly continuous semigroup $(T_t)_{t > 0}$ viz
$$G_\alpha := (L + \alpha)^{-1} \quad\text{ and }\quad T_t := e^{-tL}.$$
Both this resolvent and this semigroup are {\em markovian}, that is
for each $e \in L^2(X,m)$ with $0 \leq e  \leq 1$ the inequalities
$$0\leq \alpha G_\alpha e \leq 1 \text{ and } 0\leq T_t e \leq 1$$
hold. Thus, the resolvent and the semigroup can be extended to all
$L^p(X,m)$ spaces via monotone approximations with $L^2$-functions.
The resulting operators are contractions on $L^p(X,m)$. If not
stated otherwise we will abuse notation and write $T_t$ and
$G_\alpha$ for the extended semigroup and  resolvent on $L^p(X,m)$.
The corresponding generators on $L^p(X,m)$ will be denoted by
$L^{(p)}$  with domain $D(L^{(p)}).$

\begin{definition}[Stochastic completeness] \label{def:sc}
A Dirichlet form $Q$ is called {\em stochastically complete}  if the associated $L^\infty$ semigroup satisfies
$$T_t1 = 1, \text{ for each }t > 0.$$
\end{definition}

Stochastic completeness is equivalent to the validity of the
equality
$$\alpha G_\alpha 1 = 1,\text{ for one/each } \alpha >0.$$
For $f \in L^1_+(X,m)$, we introduce the {\em Green operator} as
$$Gf := \lim_{n \to \infty} \int_0 ^n T_s f ds,$$

where the integral is understood in the Bochner sense. Note that by
the positivity of $T_s$ this limit exists as an $m$-a.e. defined
function which might be infinite on a set of positive measure.

\begin{definition}[Recurrence/transience] \label{def:rec/trans}
 A Dirichlet form is called {\em recurrent} if for any $f \in L^1_+(X,m)$ we have  $Gf =0$ $m$-a.e. or $Gf = \infty$ $m$-a.e.  It is called {\em transient} if for each $f \in L_+^1(X,m)$ the inequality $Gf < \infty$ $m$-a.e. holds.
\end{definition}

Note that an arbitrary Dirichlet form might not be recurrent or
transient. Thus, let us recall that a measurable set $A$ is called
{\em $Q$-invariant} if for each $f \in D(Q)$ we have  $1_A f \in
D(Q)$ and the equality
$$Q(f) = Q(1_A f) + Q(1_{X \setminus A} f)$$
holds. Here $1_A$ is the indicator function of the set $A$. A
Dirichlet form $Q$ is called {\em irreducible} if each $Q$-invariant
set $A$ satisfies $m(A) = 0$ or $m(X \setminus A) = 0$. For
invariant Dirichlet forms a dichotomy holds. They are either
recurrent or transient but not both. For recent results on
irreducible Dirichlet forms we refer the reader to \cite{LSV}.

%%%%%%%%%%%%%%%%%%%%%%%%%%%%%%%%%%%%%%%%
\section{Extensions of Dirichlet forms and a maximum principle}\label{section;Extensions}
In concrete applications of Dirichlet forms one is often given two forms viz
one form corresponding to Dirichlet boundary conditions and the other
corresponding to Neumann boundary conditions.
Then, the  form with Neumann boundary conditions is an extension of the form
with Dirichlet boundary conditions.
Under suitable geometric conditions these two forms will actually agree.
In this section we provide an abstract study of such a situation.
 More precisely, we study a pair of Dirichlet forms with one form
 extending the other.
 We seek for conditions ensuring that the two forms  (or, equivalently, their
domains of definition) are actually equal.

\subsection{Hilbert space theory}

Throughout, we assume the following situation (S):

\begin{itemize}
\item[(S)]
Let $(X,m)$ be a $\sigma$-finite measure space. Let $Q$  with domain $\LD$
 and $\Qs$ with domain $\Ds$ be Dirichlet forms on $(X,m)$ such that
$\LD\subseteq \Ds$ and $Q$ and $\Qs$ agree  on $\LD$.
The  generators of $Q$ and $\Qs$ are denoted  by $L$ and $\Ls$ respectively.
 \end{itemize}

Under the assumption (S),  the inclusion
$$j: \LD\subseteq L^2 (X,m)\longrightarrow \Ds, u\mapsto u, $$
gives rise  (by taking adjoints w.r.t. $\Qs$) to the operator
 $L'$  with domain
$$D(L') =\{ u\in \Ds \,|\, \mbox{there exists $w\in L^2 (X,m)$ s.t.
$\Qs(u,v) = \langle w, v\rangle$ for all $v\in \LD$}\}$$
via
$$L' u = w.$$
The following is an immediate consequence of the definitions.
\begin{prop} \label{generator}
Assume (S). Then, $L'$ is an extension of both $L$ and $\Ls$, i.e., both the
domain of $L$ and of  $\Ls$ are contained in the domain of $L'$ and $L'$ agrees
with $L$ and $\Ls$ respectively on their domains.
\end{prop}

\begin{remark} In general the operator $L'$ will not be injective. For example,
if $X$ is a compact manifold and $Q$ is the form associated to the
Neumann-Laplacian, then $1$ will be an eigenfunction to the eigenvalues $0$ of
$L'$ (as $L'$ is an extension of the Neumann operator by the preceding
proposition).
\end{remark}

We define the space of $1$-harmonic functions by
$$\mathcal{H}:=\{u\in \Ds\, |\,  L' u = - u\}.$$
Recall that -- as $\Qs$ is a Dirichlet form --  the space $\Ds$ is a Hilbert space
with respect to the inner product
$$\langle u,v\rangle_{\Qs}:= \Qs (u,v)+ \langle u,v\rangle.$$
Here is the main result on the structure of $\Qs$ in terms of $Q$ (see Chapter
3, Section 3.3 of \cite{FOT}  for related  results).

\begin{theorem}[Decomposition theorem]  Assume (S). Then, both $\mathcal{H}$ and $\LD$ are closed
subspaces of
the Hilbert space $(\Ds, \langle\cdot,\cdot\rangle_{\Qs})$. Moreover, they are
orthogonal to each other and $$\Ds = \mathcal{H} \oplus \LD$$
holds.
\end{theorem}
\begin{proof} By construction the space $\LD$ is a closed subspace of $\Ds$.
The remaining statements follow if we show that $\mathcal{H}$ is
just the orthogonal complement of  $\LD$ in $\Ds$.   Thus, let $u\in
\mathcal{H}$ and $v\in \LD$ be given. Then, invoking the definition
of $L'$, we obtain
$$\langle{u,v \rangle}_{\Qs} =\Qs (u,v)+ \langle u, v\rangle =\langle{(L'+1)u,v}\rangle= 0.$$
Thus, $\mathcal{H}$ is orthogonal to $\LD$. Conversely, assume that
$u\in \Ds$ is orthogonal to $\LD$. Then, $0 =\langle{u,v
\rangle}_{\Qs} =\Qs (u,v)+ \langle u, v\rangle $ holds for all $v\in
\LD$. By definition of $L'$ this gives $L' u = - u$ and hence $h\in
\mathcal{H}$ follows.
\end{proof}

We note the following immediate corollary of the theorem.

\begin{coro}\label{coro;3.3}
 Assume (S). Then, the following assertions are equivalent:
 \begin{itemize}
 \item[(i)] $\mathcal{D} = \Ds$.
 \item[(ii)] There does not exist a nontrivial  $u\in D (L')$ with $L'
u = - u$.
\end{itemize}
\end{coro}

We are now heading towards studying a different aspect of validity of  $\Qs= Q$.

\begin{lemma} Assume (S). If there exists an $u\in D(L') \cap L^1
(X,m)$ with $u\geq 0$ and $L' u \leq 0$ and $ L' u \neq 0$, then
$\mathcal{D}\neq \Ds$.
\end{lemma}
\begin{proof} To simplify the argument we use the {notation}   $a < b$ for real valued measurable  functions on $a,b$ on $X$  to mean that $a(x) \leq b(x)$ for $m$ almost every $x$ and that $a$ and
$b$ differ on a set of positive $m$ measure.

Let $u\in D(L') \cap L^1 (X,m)$ with $u\geq 0$ and $L' u \leq 0$ and
$ L' u \neq 0$  be given and
assume $Q = \Qs$.  Then, by construction,  $L'$ agrees with $L$ and, in
particular, $L u = L' u < 0$ holds. As $(L +1)^{-1}$ is positivity preserving,
we then find
$$ (**)\:\;\: 0 >  (L +1)^{-1} L u = L (L +1)^{-1} u = u -  (L + 1)^{-1} u.$$
As $u$ belongs to $L^1 (X,m)$,  so does $(L + 1)^{-1} u$. As $(L + 1)^{-1}$ is a
contraction on $L^1 (X, m)$ and $u\geq 0$, we infer
$$\int_X (L + 1)^{-1} u \: dm \leq \int_X u \: dm.$$
Using this we obtain by integrating $(**)$
$$ 0 >\int_X (u - (L +1)^{-1} u) \: dm = \int_X u dm - \int_X (L +1)^{-1} u \;
dm \geq\int_X u \:dm - \int_X u \: dm = 0.$$
This is a contradiction.
\end{proof}

To provide a converse of this lemma, we will need one further concept.

\begin{definition}\label{def;MP}
 Assume (S).
Then, the \textit{maximum principle}  (MP) is said to hold if
$$ (L + 1)^{-1} f\leq (\Ls + 1)^{-1} f$$
holds for all $0 \leq f \in L^2(X,m)$.
\end{definition}

\begin{lemma} Assume (S) and (MP).  If $\mathcal{D}\neq \Ds$, then there exists
a nontrivial  $u\in  L^1 (X,m) \cap L^\infty (X,m) \cap \mathcal{H}
$ with $u\geq 0$ and $L' u \in L^1 (X,m)$. In particular,  such a
function  $u$ satisfies the inequality
$$\int_X L' u \, dm \neq 0.$$
\end{lemma}
\begin{proof}  By $\mathcal{D}\neq \Ds$ the operators $(L+1)^{-1} $ and $(\Ls+1)^{-1}$ are different. As $L^1 (X,m)\cap L^\infty (X,m)$ is dense in $L^2(X,m)$, there must  exist an $f\in L^1 (X,m)\cap L^\infty (X,m)$ with $f\geq
0$ and
$$ 0\neq (\Ls + 1)^{-1} f  - (L + 1)^{-1} f=:u.$$
Then, $u\geq 0$ holds  by (MP).
As $Q$ and $\Qs$ are Dirichlet forms and $f$ belongs to $L^1 (X,m)\cap L^\infty(X,m)$, we infer that $u$ belongs to $L^1 (X,m)\cap L^\infty (X,m)$. Moreover,
by Proposition \ref{generator} we have
$$(L' +1)  u =  f - f =0$$
and, hence,
$$L' u = - u\in L^1 (X,m).$$
As $u\geq 0$ with $u\neq 0$ holds, this shows  the first statement. As for the
last statement, we note that
 obviously
$$\int_X L' u \; dm = - \int_X  u \; dm < 0$$
holds.
This finishes the proof.
\end{proof}

Combining the previous two lemmas we immediately infer the following theorem.

\begin{theorem}[Characterization of $Q = \Qs$] \label{thm;2.2}
Assume (S) and (MP). Then, the following assertions are equivalent:
\begin{itemize}
\item[(i)] $\mathcal{D}\neq \Ds$.
\item[(ii)] There exists a nontrivial  $u\in L^1 (X,m)\cap D
(L')$ with $u\geq 0$ and $L' u \leq 0$ and $L' u \neq 0$.
\end{itemize}
If the assertions hold, then in $\mathrm{(ii)}$ the function $u$ can
be chosen in $L^1 (X,m)\cap L^\infty (X,m)$. Such a function  $u$
satisfies the inequality
$$\int_X L'u\, dm \neq 0. $$
\end{theorem}

\begin{remark} \label{ex situation} Let $X$ be a locally compact,
$\sigma$-compact topological space. Let $\Qs$ with domain $\Ds$ be a Dirichlet
form on $X$ such that $\Ds \cap C_c (X)$ is dense in $C_c (X)$ (with respect to
the supremum norm). Define $Q$ to be the closure of the
restriction of $\Qs$ to $\Ds\cap C_c (X)$. Then, $Q$ is a regular Dirichlet form
and $Q$ and $\Qs$ form a pair of Dirichlet forms satisfying (S). In such a
situation  (MP) is often known to hold  (e.g., for manifolds, metric graphs and graphs)
and our main result can be applied. For a detailed discussion see Section~\ref{section;MP} and Proposition~\ref{MP regular}.
\end{remark}

\subsection{An approximation criterion for the maximum principle } \label{section;MP}

In this section we will prove the maximum principle (MP) (see Definition~\ref{def;MP})
for the situation that the resolvents can be approximated via restrictions to an exhausting sequence $(G_n)$ in $X$.
For this some preparations are needed.
Let $(Q,D(Q))$ be  a Dirichlet form. For any measurable $G \subseteq X$, we
consider the form $Q_{G}$ given by
$$D(Q_{G}) = \{u \in D(Q)\, |\, u \equiv 0\ m\text{-a.e.\ on } X \setminus G\},
\quad Q_{G}(u)= Q(u).$$
Then, $(Q_{G},D(Q_{G}))$ is a Dirichlet form in the wide sense on $L^2(X,m)$, i.e., $D(Q_{G})$ is not necessarily dense in $L^2(X,m)$, see  \cite[Theorem 1.3.2]{FOT} and the discussion preceding it. The associated resolvent on $L^2(X,m)$, denoted by
$(L_{G}+\alpha)^{-1}$, may not be strongly continuous and, thus, may not give rise to a densely defined self-adjoint operator. However, below we work with the resolvent and the form only. The following proposition is taken from \cite{SV}. We include a proof for the convenience of the reader.

\begin{prop}\label{Domain monotonicity} Assume the situation described above.
For all nonnegative $f \in L^2(X,m)$ the following estimate holds
$$(L_G+\alpha)^{-1}f \leq (L+\alpha)^{-1}f .$$
\end{prop}
\begin{proof}
To simplify notation we only consider the case $\alpha
= 1$. The case of general $\alpha > 0$ can be treated similarly. Let $P$ be the orthogonal projection of $D(Q)$ onto $D(Q_G)$
in $(D(Q),\as{\cdot,\cdot}_Q)$ and let $0\leq f \in L^2(X,m)$. Then,
$$(L_G+1)^{-1}f = P(L+1)^{-1}f$$
since for $u \in D(Q_G)$ the equation
$$\langle(L+1)^{-1}f,u \rangle_Q = \langle f,u \rangle = \langle(L_G+1)^{-1}f,u
\rangle_{Q_G}$$
holds. Resolvents of Dirichlet forms are positivity preserving. Therefore, the
above implies that showing $Pu \leq u$ for all positive $u \in D(Q)$ settles
the claim. Given  $u \geq 0$, the function $u \wedge Pu$ belongs to $D(Q_G)$.
Thus, we conclude
\begin{align*}
\| u - u\wedge Pu\|_Q^2 &=  \| (u -  Pu)_+\|_Q^2 = Q( (u -  Pu)_+) + \|(u -
Pu)_+\|^2\\
&\leq  Q( u -  Pu) + \|u -  Pu\|^2 = \| u -  Pu\|_Q^2.
\end{align*}
Since $Pu$ is the unique distance minimizing element in $D(Q_G),$ we obtain $Pu
= u\wedge Pu$. This finishes the proof.
\end{proof}

With the proposition at hand, we can prove an approximation result for general Dirichlet forms.

\begin{prop} \label{appr resolvent}
Let $(Q,D(Q))$ be a Dirichlet form on $(X,m)$. Suppose $(G_n)$ is an
increasing sequence of subsets of $X$ and set
$$\mathcal{C} = \bigcup_{n=1}^\infty D(Q_{G_n}).$$
Let $Q_\mathcal{C}$ be the restriction of $Q$ to the closure of $\mathcal{C}$
in $(D(Q),\aV{\cdot}_Q)$ and let $(L_{\mathcal{C}}+\alpha)^{-1}$ be the
associated resolvent. Then, for any $f \in L^2(X,m)$ and $\alpha >0$,
$$(L_{G_n} + \alpha)^{-1}f \to (L_{\mathcal{C}} + \alpha)^{-1}f \text{ as }
n\to \infty.$$
For nonnegative $f$ this convergence is monotone. In particular, if
$\mathcal{C}$ is dense in $D(Q)$, then the resolvent of $Q$ can be approximated as
above.
\end{prop}

\begin{proof}
After decomposing $f$ into positive and negative part, we can restrict our attention to $f\geq 0$. As the norm of $(L_{G_n}+\alpha)^{-1}$ is uniformly
bounded by $\frac{1}{\alpha}$ and $L^2\cap L^\infty$ is dense in $L^2$, we may assume that $f$ is bounded.
We set $u_n:= (L_{G_n}+\alpha)^{-1}f $ and notice $u_{n}\in \mathcal{C}$, $n\ge1$.

We  first show that $u_n$ converges to a function $u \in D(Q_\mathcal{C})$. Proposition~\ref{Domain monotonicity} implies that the
sequence $u_n$ is $m$-a.e.\ monotone increasing. Furthermore, standard Dirichlet
form theory implies  $0 \leq u_n \leq \frac{1}{\alpha}\|f\|_\infty$. This shows
that $u_n$ is almost surely convergent to a bounded function $u$. The
construction of resolvents from forms yields $\|u_n\|\leq
\frac{1}{\alpha}\|f\|$. Therefore, the convergence $u_n \to u$ also holds in
$L^2(X,m)$. Let us compute for $n\geq m$
\begin{align*}
\lefteqn{Q(u_n-u_m) + \alpha \|u_n-u_m\|^2}
\\
&= Q(u_n) + \alpha \|u_n\|^2 + Q(u_m) + \alpha \|u_m\|^2 - 2 ( Q(u_n,u_m) +
\alpha \langle u_n,u_m \rangle ) \\
&= Q(u_n) + \alpha \|u_n\|^2 + Q(u_m) + \alpha \|u_m\|^2  - 2\langle
f,u_m\rangle\\
&= Q(u_n) + \alpha \|u_n\|^2 - \langle  f,u_m\rangle \\
&= \langle f, u_n - u_m \rangle \\
& \leq \|f\|^2 \|u_n-u_m\|^2 \to 0 \text{ as } n,m \to \infty.
\end{align*}
This shows that $(u_n)$ is a Cauchy sequence in $D(Q_\mathcal{C})$.
Since $D(Q_{\mathcal{C}})$ is complete, we conclude $u \in
D(Q_\mathcal{C})$ and $u_n \to u$ in $D(Q_{\mathcal{C}})$. Next, we
prove that $u$ is a  minimizer of
$$q:D(Q_\mathcal{C})\longrightarrow [0,\infty),\quad
v\mapsto Q(v) + \alpha\|v - \frac{1}{\alpha}f\|^2.$$
This  yields the statement since such a minimizer is known to be
unique and to agree with the resolvent by a variational
characterization of resolvents: Let $w \in \mathcal{C}$ be
arbitrary. Then, there exists an $n_0$ such that $w \in D(Q_{G_n})$
for all $n \geq n_0$. Since the $u_n$ are also resolvents and, thus,
minimizers of $q$ on $D(Q_{G_n})$ we obtain
\begin{align*}
Q(u) + \alpha\|u - \frac{1}{\alpha}f\|^2 &= \lim_{n\to \infty} Q(u_n) +
\alpha\|u_n - \frac{1}{\alpha}f\|^2 \leq  Q(w) + \alpha\|w - \frac{1}{\alpha}f\|^2.
\end{align*}
Since $w \in \mathcal{C}$ was arbitrary and $\mathcal{C}$ is dense in
$D(Q_\mathcal{C})$, we obtain the statement.
\end{proof}

\begin{remark}
With the above theorem we established Mosco convergence of the forms
$Q_{G_n}$ to $Q_\mathcal{C}$ by simple monotonicity arguments. This
idea may  already be found in the proof of Proposition 2.7 in
\cite{KL}. Instead we could have also used a more abstract
characterization of this convergence. See the Appendix of
\cite{CKK}.
\end{remark}

\begin{theorem}[Sufficient condition for maximum principle]\label{max principle} Assume  the forms $(Q,D(Q))$ and $(\Qs,D(\Qs))$ satisfy $(S)$. Let $(G_n)$ be an increasing sequence of subsets
of $X$ and let
$$\mathcal{C} = \bigcup_{n=1}^\infty D(\Qs_{G_n}).$$
Assume   $\mathcal{C} \subseteq D(Q)$ and that the closure of $\mathcal{C}$ coincides with
$D(Q)$. Then, $Q$ and $\Qs$ satisfy the maximum principle (MP).
\end{theorem}

\begin{proof}
Our assumptions and Proposition \ref{appr resolvent} imply that
$(\Ls_{G_n}+\alpha)^{-1}$ converge strongly to $(L+\alpha)^{-1}$.
Furthermore, Proposition~\ref{Domain monotonicity} shows that for nonnegative
$f\in L^2(X,m)$ this convergence is monotone and
$$(\Ls_{G_n}+\alpha)^{-1}f\leq (\Ls  +\alpha)^{-1}f.$$
This proves the claim.
\end{proof}

%%%%%%%%%%%%%%%%%%%%%%%%%%%%%%%%%%%%%%%%
%%%%%%%%%%%%%%%%%%%%%%%%%%%%%%%%%%%%%%%%
\section{Stochastic completeness}\label{section;SC}
In this section we  present a characterization of stochastic completeness (see Definition \ref{def:rec/trans}) in Theorem~\ref{thm;SC}. We give a proof by three lemmas
which immediately imply the theorem.

\begin{theorem}[Characterization of stochastic completeness]  \label{thm;SC}
Let $Q$ be a Dirichlet form  with associated self-adjoint operator
$L$. Let $L^{(1)}$ with domain $D(L^{(1)})$ be the generator of the
$L^1$-semigroup  associated to $Q$. Let $L^{(\infty)}$ be the
adjoint of $L^{(1)}$. Then, the following assertions are equivalent:
\begin{itemize}
\item[(i)]  $Q$ is stochastically complete.
\item[(ii)] For all $u\in D(L^{(1)})$ the following equality holds
$$\int_X L^{(1)} u \; dm =0.$$
\item[(iii)]  The constant function $1$ belongs to the domain of $L^{(\infty)}$
and $L^{(\infty)} 1  =0$ holds.
\item[(iv)] For all $u \in D(L) \cap L^1(X,m)$ such that $Lu \in L^1(X,m)$
the following equality holds
$$\int_X L u \; dm =0.$$
\end{itemize}
\end{theorem}
 The proof  follows from the subsequent three lemmas.

\begin{remark}
\begin{itemize}
 \item In the context of stochastic completeness, (iv) is already
discussed in the literature for weighted manifolds \cite{GM}.
\item It can be shown that the adjoint of $L^{(1)}$ is indeed the generator of the $L^\infty$ resolvent which is associated with $Q$.
\end{itemize}

\end{remark}

The following lemma is certainly well known. For an $L^1$-Version see  \cite[Prop.2.4.2]{BH}. It may be useful in other contexts as well. We include
a proof for the convenience of the reader.

\begin{lemma}\label{SC-one} Let $Q$ be a Dirichlet form and let $L$ be the associated self-adjoint operator. For $p\in [1,\infty)$,  let
$L^{(p)}$ with domain $D(L^{(p)})$ be the generator of the  $L^p$-semigroup  associated to $Q$. Then,
$$ S:=\{ f\in  L^p (X,m) \cap D (L)\, | \,  L u \in L^p
(X,m)\}$$
is contained in $ D (L)\cap D (L^{(p)})$
and $L$ agrees with $L^{(p)}$ on $S$ and to any $f\in  D(L^{(p)})$
there exists a sequence $(f_n)$ in $S$ with $f_n\to f$ in $L^p (X,m) $ and $L
f_n \to L^{(p)}f$ in $L^p (X,m)$.
\end{lemma}
\begin{proof} We first show $S\subset D(L^{(p)})$ and $Lf = L^{(p)} f$
for $f\in S$, (which implies the inclusion $S\subset   D (L)\cap D
(L^{(p)})$): For $f\in S$ the function $g:= (L +1) f$ belongs to
$L^p (X,m)\cap L^2 (X,m)$. Thus, we can  apply both $(L+1)^{-1}$ and
$(L^{(p)} +1)^{-1} $ to $g$ and by consistency of the resolvents we
obtain
$$ f = (L+1)^{-1} g= (L^{(p)} +1)^{-1} g.$$
This gives  $f\in D(L^{(p)})$ as well as
$$ L f = (L +1 -1) (L+1)^{-1} g = g - (L+1)^{-1} g = g - (L^{(p)} +1)^{-1}g =
L^{(p)} f.$$

We now show  denseness:  Let $f\in D(L^{(p)})$ be given.
Then, $g := (L^{(p)} +1) f$ exists and $f = (L^{(p)} +1)^{-1}g$ holds.  By
$\sigma$-finiteness we can  choose a sequence $g_n \in L^p (X,m)\cap L^2
(X,m)$ with $g_n \to g$ in $L^p (X,m)$.  By continuity of resolvents we then
have
$$ f_n :=(L^{(p)} +1)^{-1} g_n\to (L^{(p)} +1)^{-1} g = f$$
in $L^p (X,m)$.  By consistency of resolvents we furthermore obtain
$$ f_n= (L+1)^{-1} g_n = (L^{(p)} +1)^{-1} g_n \in D (L) \cap
D(L^{(p)})$$
and
$$L f_n =   g_n - f_n = L^{(p)} f_n. $$
Putting these statements together we infer
$$ L f_n = g_n - f_n \to g - f = L^{(p)} f$$
in $L^p (X,m)$. This finishes the proof.
\end{proof}

\begin{lemma} \label{SC-two} Let $Q$ be a Dirichlet form  with  associated self-adjoint operator $L$. Then, the following
assertions are equivalent:
\begin{itemize}
\item[(i)]  $Q$ is stochastically complete.
\item[(ii)] For all $u\in D(L)\cap L^1 (X,m)$ with $L u\in L^1 (X,m)$
the following equality holds
$$\int_X Lu \; dm =0.$$
\end{itemize}
\end{lemma}
\begin{proof} (i) $\Longrightarrow$ (ii): We choose a sequence $(g_n)
\subset L^2(X,m)$ such that $0\leq g_n \leq g_{n+1}\leq 1$ and  $g_n \to 1$
$m$-almost everywhere. Furthermore, let $e_n = (L+1)^{-1}g_n$. By stochastic
completeness $(e_n)$ converges to $1$ $m$-almost everywhere. Thus, for any $u \in
D(Q)\cap L^1(X,m)$, we obtain by Lebesgue's theorem
$$\lim_{n\to \infty} Q(e_n,u) = \lim_{n\to \infty}\langle g_n - e_n,u  \rangle
= 0.$$ For $u$ that satisfies  additionally $u\in D(L)$ and $Lu \in
L^1 (X,m)$, we
 obtain by Lebesgue theorem
\begin{eqnarray*}
0 = \lim_{n\to \infty} Q (e_n, u)
= \lim_{n\to \infty} \langle e_n, L u\rangle
=\lim_{n\to \infty} \int_X e_n L u dm
= \int_X Lu \: dm.
\end{eqnarray*}
This finishes the proof of this implication.

(ii) $\Longrightarrow$ (i): By $\sigma$-finiteness of $(X,m)$ and Proposition~\ref{choice-e-n} there exists  a sequence $(e_n)$ in $D(Q)$ with $0\leq e_n
\leq 1$ and $e_n \to 1$ $m$-almost surely. Now, choose an arbitrary  $f\in L^1
(X,m)\cap L^2 (X,m)$ with $f>0$. (Such a choice is possible by
$\sigma$-finiteness). Set $v:= (L +1)^{-1} f$. By construction,  $v$ belongs to
the domain of $L$. Moreover, as $(L+1)^{-1}$ is a contraction on $L^p (X,m)$
for any $p\geq 1$ and $f$ belongs to $L^1 (X,m)\cap L^2 (X,m)$,  we infer $v\in
L^1 (X,m)\cap L^2 (X,m)$. Furthermore, we obviously have
$$L v =  (L+1 - 1) v= f - v\in L^1 (X,m).$$
Thus, by (ii) we then obtain
$$\int_X L v \; dm = 0.$$
This gives
\begin{align*}
0 = \int_X L v  dm
= \lim_{n\to \infty}\int_X e_n L v   dm
= \lim_{n\to \infty} \langle e_n, f - (L+1)^{-1}f \rangle
= \lim_{n\to \infty} \langle e_n - (L+1)^{-1}e_n, f \rangle.
\end{align*}
Since $f$ was chosen strictly positive, this yields $(L+1)^{-1} 1 = 1$ which is
equivalent to stochastic completeness.
\end{proof}

\begin{remark}
The above proof shows that stochastic completeness is equivalent to the
existence of a sequence $(e_n) $ in $ D(Q)$ satisfying $0\leq e_n \leq 1$,
$e_n \to 1$ $m$-almost everywhere and
$$  \lim_{n\to \infty} Q(e_n,u) = 0,$$
for any $u \in D(Q)\cap L^1(X,m)$. This part of the proof is taken from Theorem
1.6.6 of \cite{FOT}.
\end{remark}

\begin{lemma}\label{SC-three} Let $Q$ be a Dirichlet form. Let $L^{(1)}$ with domain $D(L^{(1)})$ be the generator of the  $L^1$-semigroup  associated to $Q$.
Let $L^{(\infty)}$ be the adjoint of $L^{(1)}$. Then, the following assertions
are equivalent:
\begin{itemize}
\item[(i)]  For all $u\in D(L^{(1)})$ the equality
$\int_X L^{(1)} u \; dm =0$
holds.
\item[(ii)]  The constant function $1$ belongs to the domain of $L^{(\infty)}$
and $L^{(\infty)} 1  =0$ holds.
\end{itemize}
\end{lemma}
\begin{proof} This is immediate from the definitions.
\end{proof}

\begin{proof}[Proof of Theorem 1] The equivalence of (i) and (iv) is shown
in Lemma \ref{SC-two} and (ii) $\Longleftrightarrow$ (iv) follows from Lemma
\ref{SC-one}. Finally, the equivalence of (ii) and (iii) is given by Lemma
\ref{SC-three}.
\end{proof}

\begin{remark} We proved the equivalence of (i) and (ii) in Theorem
\ref{thm;SC} via (iv) and a denseness argument. Using semigroup
theory one could also proceed as follows: Let $(T_t^{(1)})$ denote
the $L^1$-semigroup associated with $Q$ and let $(T^{(\infty)}_t)$
be its adjoint. For $u \in D(L^{(1)})$ the function $T_t^{(1)}u$ is
a solution to the heat equation on $L^1$. Therefore,
\begin{align*}
\int_X L^{(1)}u \,dm  &= -\frac{d}{dt}\int_X T_t^{(1)}u \, dm  \\
&= -\frac{d}{dt}\int_X uT^{(\infty)}_t 1\, dm.\\
\end{align*}
If $Q$ is stochastically complete, then the right hand side of the
above equation vanishes and (ii) follows. If (ii) holds the above
shows that $\int_X uT^{(\infty)}_t 1dm$ is constant for all $u \in
D(L^{(1)})$. This then easily  implies $T^{(\infty)}_t 1 = 1$.
\end{remark}

%%%%%%%%%%%%%%%%%%%%%%%%%%%%%%%%%%%%%%%%%%%%%%%%%%%%%%%%%%

\section{Recurrence}\label{section;recurrence}
In this section we  characterize  recurrence  of Dirichlet forms (see Definition~\ref{def:rec/trans}). The
crucial new ingredient in our considerations will be the operator $L_e$ defined
below.

To each Dirichlet form with domain $D(Q)$ we can associate the
extended Dirichlet space $D(Q)_e$ which consists of all $m$-a.e.\
finite measurable functions $f$ for which a $Q$-Cauchy sequence
$(f_n) \subseteq D(Q)$ exists such that $f_n \to f$ $m$-a.e.\ Such a
sequence is called \textit{approximating sequence} for $f$. We can
then extend $Q$ to a quadratic form on $D(Q)_e$ by setting
$$Q(f) = \lim_{n\to \infty} Q(f_n).$$\label{definition;extended dirichlet space}
For properties of this space and further details, we refer the
reader to \cite[Chapter~1]{CF}. Note that we denote the extended
form by $Q$ as well.

The inclusion
$$j : D(Q)\subseteq L^2 (X,m) \to D(Q)_e, \quad u\mapsto u,$$
gives (via taking the adjoint)  rise to the operator
$$L_e : D(Q)_e\to L^2 (X,m)$$
with domain
$$D(L_e):=\{v\in D(Q)_e\mid \mbox{there exists  $w\in L^2 (X,m)$
s.t. $\langle w, u\rangle = Q (v,u)$ for all $u\in D(Q)$}\}$$
acting by
$$L_e v = w.$$

\begin{remark}
It is not hard to see that $L_e$ is an extension of the operator $L$
associated to $Q$ in the sense that  $D(L)\subseteq D(L_e)$  and
$L_e f = L f$ for $f\in D(L)$.
\end{remark}

The proof of the next result strongly relies on results of \cite{CF}.

\begin{theorem}[Characterization recurrence] \label{thm;recurrence} Let $Q$ be an irreducible Dirichlet form.  Then the following assertions are
equivalent:
\begin{itemize}
\item[(i)] $Q$ is recurrent.
\item[(ii)] For all $u\in D(L_e)$ with $L_e u\in L^1 (X,m)$ the following
equality holds
$$0 = \int_X L_e u \;  dm .$$
\item[(iii) ] The constant function $1$ belongs to the domain of $L_e$ and $L_e
1 =0$ holds.
\end{itemize}
\end{theorem}
\begin{proof} (i) $\Longrightarrow$ (ii): By  assumption (i) and   \cite[Part (ii)
of Theorem~2.1.8]{CF}, there exists a sequence $(e_n)$ in $D (Q)$ with
$0\leq e_n \leq 1$ and $e_n \to 1$ $m$-almost everywhere and
$$0 = \lim_{n\to \infty} Q (e_n, u)$$
for all $u\in D(Q)_e$. For $u$ satisfying additionally $u\in D(L_e)$
and $L_eu \in L^1 (X,m)$ we  obtain by Lebesgue's theorem
\begin{eqnarray*}
0 = \lim_{n\to \infty} Q (e_n, u)
= \lim_{n\to \infty} \langle e_n, L_e u\rangle
= \lim_{n\to \infty} \int_X e_n L_e u\;dm
= \int_X L_e u \: dm.
\end{eqnarray*}
This finishes the proof of this implication.

(ii) $\Longrightarrow$ (i): Because of the irreducibility of $Q$, it suffices to
show that transience implies the existence of a  $u\in D(L_e)$ with
$L_e u \in L^1 (X,m)$ and
$$0 \neq \int_X L_e u \; dm. $$
By transience of $Q$ the space $D(Q)_e$ with the inner product $Q$ is a Hilbert
space  and  there exists a strictly positive $g\in L^1 (X,m)$ with
\begin{align*}\tag{$*$}
\int_X |v| g dm \leq Q (v)^{1/2}
\end{align*}
for all $v\in D(Q)_e$ (see  \cite[Theorem 2.1.5]{CF}).  Without loss of
generality we can assume $g\in L^2 (X,m)$ as well.
The functional
$$F_g : D(Q)_e \to \RR,\quad  v\mapsto \int_X g v\, dm,$$
is continuous by $(*)$.  Thus, by Riesz representation theorem there exists
$u\in D(Q)_e$ with
$$ \langle g,v\rangle =  \int g v\, dm = F_g (v) = Q (u,v)$$
for all $v\in D(Q)$. By definition of $L_e$ this implies
$$ L_e u = g.$$
As $g$ is strictly positive, we obtain $\int_X L_e u\:  dm >0$. This is the
desired statement.

(i) $\Longleftrightarrow $ (iii): By \cite[Theorem~2.1.8]{CF}
recurrence is equivalent to the function  $1$ belonging to $D(Q)_e$.
In this case, one has $Q (1,u) = 0$ for all $u\in D(Q)$
(see \cite{CF} as well). This gives the desired equivalence.
\end{proof}

\begin{remark}
\begin{itemize}
\item
 The irreducibility of $Q$ is needed in the previous theorem to ensure the
dichotomy of recurrence and transience in our context (see
\cite[Lemma 1.6.4]{FOT}).
\item To put condition (iii) in perspective we define
$$(L_e)_1 : \{u\in D (L_e)\, | \, L_e u \in L^1 (X,m)\}\to
L^1 (X,m), u\mapsto L_e u.$$
Then, (ii) is equivalent to $1\in D ((L_e)_1^\ast)$ and $(L_e)_1^\ast
1 =0$. In this sense, (iii) can be understood as some  form of ``symmetry'' of
$(L_e)_1$.
\end{itemize}
\end{remark}

%%%%%%%%%%%%%%%%%%%%%%%%%%%%%%%%%%%%%%%%%%%%%%%%%%%%%%%%%%%%%%%%%%%%%%%%%%%%%%%
%%%%%%%%%%%%%%%%%%%%%%%%%%%%%%%%%%%%%%%%%%%%%%%%%%%%%%%%%%%%%%%%%%%%%%%%%%%%%%%%
%

\section{Application to regular Dirichlet forms}\label{section;Regular}
In this section we apply the theory  developed so far to a regular
Dirichlet form. Using the Beurling-Deny decomposition each such form
can be extended to the so called reflected Dirichlet space.  Using
this  space we can then provide a  unified treatment of all the
operators and spaces which were used above. In fact, we will show
that all the operators above are just  restrictions of a single
operator to suitable domains.

Let $X$ be a locally compact separable metric space and $m$ a Radon measure of full support. We consider a regular Dirichlet form $(Q,D(Q))$ on $L^2(X,m)$, where regular means that $C_c(X) \cap D(Q)$ is dense in $D(Q)$ with respect to $\aV{\cdot}_Q$ and in $C_c(X)$ with respect to $\aV{\cdot}_\infty$.

A function $f:X\to\RR$ is said to be \emph{quasi  continuous} if for every $\eps>0$ there is an open set $U\subseteq X$ with capacity
less than $\eps$, i.e.,
\begin{align*}
    \mathrm{cap}(U):=\inf\{\|v\|_{Q}\mid v\in D(Q),\, 1_{U}\leq v\}\leq \eps,
\end{align*}
such that $f\vert _{X\setminus U}$ is continuous (where
$\inf\emptyset=\infty$ and $1_{U}$ is the characteristic function of $U$). For a regular Dirichlet form $Q$ every $u\in D(Q)$ admits a
quasi continuous representative, see \cite[Theorem~2.1.3]{FOT}.
Moreover, we say a function satisfies a property \emph{quasi everywhere}, q.e., if the property holds outside of a set $N\subseteq X$ of
capacity zero, i.e., $\mathrm{cap}(N)=0$, where
${\rm cap}(A)=\inf\{{\rm cap}(U)\mid A\subseteq U \mbox{ open}\}$ for
$A\subseteq X$.

We can  express the regular Dirichlet form $Q$ using the
\emph{Beurling-Deny formula}  \cite[Theorem~3.2.1]{FOT}.
That is, for any function $u$ belonging to $D(Q)$ the equation
\begin{equation*} \label{beurling deny}
Q(u,u)=\int_X d\muc (u) +
\iint_{X \times X \setminus {\it diag}} (\tilde u (x)
- \tilde u (y))^2\, J(dx,dy)
     +\int_X \tilde u^2 dk
\end{equation*}
holds. Here $\muc (u)$ is the \textit{strongly local measure},
$J(dx,dy)$ is the \textit{jump measure}, ${\it diag}$ is the
diagonal set of $X \times X$, and $k$ is the \textit{killing
measure} associated with $Q$ (see, e.g., \cite[Chapter~3]{FOT} for
construction and properties of these objects). Furthermore, $\tilde
u$ denotes a quasi-continuous representative of $u$. We denote by
$$D(Q)_{\rm loc} = \{u \in L^2_{\rm loc} \,|\, \forall G \subseteq X \text{
open, relatively compact } \exists v \in D(Q) \text { with } u = v \text{ on }
G \}$$
the space of functions locally belonging to the domain $D(Q)$. Note that also each $u
\in D(Q)_{\mathrm{loc}}$  admits a quasi continuous representative $\tilde{u}$ (see,
e.g., \cite[Proposition~3.1]{FLW}).

The Beurling-Deny representation allows one to extend the diagonal
of $Q$ to larger classes of functions. Since any function in
$D(Q)_{\mathrm{loc}}$ has a quasi continuous representative and $J$
and $k$ charge no set of capacity zero, we can extend the second and
third summand of the Beurling-Deny formula to $D(Q)_{\mathrm{loc}}$
in an obvious way. Moreover, for $u \in D(Q)_{\mathrm{loc}}$ we
introduce the Radon measure $d\mu^{(c)}(u)$ via the identity
$$\int_X \varphi \,d\mu^{(c)}(u) = \int_X \varphi \,d\mu^{(c)}(u_\varphi), $$
where $\varphi \in C_c(X)$ and $u_\varphi \in D(Q)$ are  such that
$u = u_\varphi$ on a neighborhood $\supp \,\varphi$. The local
property of $\mu ^{(c)}$ assures that this is well defined. Thus,
the diagonal of the form $Q$ can be extended to
$D(Q)_{\mathrm{loc}}$. We will denote this extension by the same
symbol $Q$ (note that it may take the value $\infty).$

Let us mention some properties of this extension.

\begin{prop} \label{some properties}  Let $u \in D(Q)_{\rm loc}$ be given.
\begin{itemize}
\item[(a)] For any normal contraction $C$ the function $C\circ u$ belongs to $
D(Q)_{\rm loc}$ and $Q(C\circ u) \leq Q(u)$ holds.
\item[(b)] The sequence $(Q((u\wedge n)\vee (-n)))_n$ is monotone
increasing with  $Q(u) = \lim_{n\to \infty} Q((u\wedge n)\vee
(-n))$.
\item[(c)] If $u$ belongs to $L^\infty(X,m)$ and has compact support, then $u \in D(Q)$.
\item[(d)] If $v \in D(Q)_{\rm loc}$ the following inequalities hold
$$Q(u \wedge v)^{1/2} \leq Q(u)^{1/2} + Q(v)^{1/2}\quad \text{ and } \quad Q(u \vee v)^{1/2} \leq Q(u)^{1/2} + Q(v)^{1/2}.$$
\end{itemize}
\end{prop}
\begin{proof}
Assertions (a) and (b) follow immediately from the definition of a regular Dirichlet form and its Beurling-Deny decomposition. Let us turn to statement statement (c). By the discussion in \cite[Section 3.1, page 4772]{FLW}, the space $L^\infty(X,m) \cap D(Q)_{\rm loc}$ is included in the space refered to as $\mathcal{D}_{loc}^*$ in \cite{FLW}. By \cite[Theorem 3.5]{FLW} the compactly supported functions in $\mathcal{D}_{loc}^*$ belong to $D(Q)$. The proof of the first inequality of (d) uses $u \wedge v = \frac{1}{2} (u+v + |u-v|)$, the triangle inequality for $Q$ and the contraction property (a). More precisely, we estimate
$$Q(u\wedge v)^{1/2} = \frac{1}{2}Q(u + v + |u-v|)^{1/2} \leq \frac{1}{2}(Q(u + v)^{1/2} + Q(|u-v|)^{1/2}) \leq Q(u)^{1/2} + Q(v)^{1/2}.$$
The other inequality can be treated similarly. This finishes the proof.
\end{proof}

%%%%%%%%%%%%%%%%%%%%%%%%%%%%%%%%%%%%%%%%%%%%%%%%%%%%%%%%%%%%%%%%%%%%%%%%

\subsection{Functions of finite energy}
\label{subsection;FFE} In this section we introduce another space of
importance. When  equipped with the extension of the underlying
Dirichlet form this space is referred to as the \textit{reflected
Dirichlet space}.

Recall that $L^0(X,m)$ denotes the space of $m$-a.e.\ defined functions. For
$n \geq 1$ and $u \in L^0(X,m)$, we write $u^{(n)} = (u\wedge n)\vee
(-n)$. Similarly, for nonnegative $f\in L^0(X,m)$, we let $u^f =
(u\wedge f)\vee (-f)$. We extend $Q$ to
$$D(Q)_{\rm loc}^\infty := \{u \in L^0(X,m)\, | \, u^{(n)} \in D(Q)_{\rm loc}
\text{ for all } n\geq 1\},$$
by setting
$$\ow{Q}(u) := \lim_{n \to \infty}Q(u^{(n)}).$$
Here,  the preceding  limit exists  as $(Q(u^{(n)}))_n$ is monotone.
Indeed, this monotonicity can be directly inferred from
 Proposition~\ref{some properties}~(b) as $(u^{(n)})^{(k)} = u^{(k)}$ for all $k\leq n$.
 Whenever $u  \in L^0(X,m) \setminus D(Q)_{\rm loc}^\infty$,
we let $\ow{Q}(u)  = \infty$.

\begin{definition}[Functions of finite energy]  We say
$$\widetilde{D}(Q) := \{u \in D(Q)_{\rm loc}^\infty\, | \, \ow{Q}(u) < \infty \}$$
is the space of \emph{functions of finite energy} associated with
$Q$. The pair $(\ow{Q},\widetilde{D}(Q))$ is called its \emph{reflected
Dirichlet space}.
\end{definition}
\begin{remark}
It is immediate from the definitions that
$$\ow{D}(Q) \cap L^\infty(X,m) = \{u\in D(Q)_{\rm loc} \cap L^\infty(X,m)\, \mid \, Q(u) < \infty\}$$
and that  $\ow{Q}$ and $Q$ agree on this space. In fact, by Proposition~\ref{some properties}~(b) the inclusion
$$\{u\in D(Q)_{\rm loc}\mid  Q(u) < \infty \} \subseteq \ow{D}(Q) $$
holds and if $u \in  D(Q)_{\rm loc}$ with $Q(u) < \infty$ the equality $\ow{Q}(u) = Q(u)$ is satisfied. Note however, that the inclusion $D(Q)_{\rm loc} \subseteq L^2_{\rm loc}(X,m)$ is satisfied by the definition of $D(Q)_{\rm loc}$ while the same need not be true for $\ow{D}(Q)$.
\end{remark}

We will now prove two structural theorems about the space $(\ow{Q},\ow{D}(Q))$. Namely, we show that $\ow{Q}$ has the Fatou property on $L^0(X,m)$ and that it is a quadratic form satisfying the Markov property.

\begin{theorem}[Fatou's Lemma for $\ow{Q}$ on $L^0$] \label{Fatou} Let $(u_n)$ be a sequence in
$L^0(X,m)$ and $u \in L^0(X,m)$ such that $u_n \to u$ $m$-almost everywhere. Then,
$$\ow{Q}(u) \leq \liminf_{n\to \infty}\ow{Q}(u_n).$$
In particular, $\liminf_{n\to \infty}\ow{Q}(u_n) < \infty$ implies $u \in
\ow{D}(Q)$.
\end{theorem}
\begin{proof}
It suffices to consider the case $\liminf_{n\to \infty}\ow{Q}(u_n) < \infty$. So, assume $u_n \in \ow{D}(Q)$ for all $n$ and $\liminf_{n\to \infty}\ow{Q}(u_n)
= \lim_{n \to \infty}\ow{Q}(u_n)$.

We prove the statement in two steps. First we show the statement for
bounded functions and conclude the general statement afterwards.

 Step 1: Assume $u\in L^\infty(X,m)$.  Without loss of generality we assume $-1 \leq u \leq 1$. Using
Proposition \ref{some properties}~(a) we may cut-off the $u_n$ and assume  $-1 \leq u_n\leq 1$ as well.

We show $u\in D(Q)_{\rm loc}$:   Let $G$ be open and relatively
compact. By regularity of $Q$ we choose a function $e \in D(Q)\cap
C_c(X)$ such that $e\equiv 1$ on $G$. Set $u_n^e = (u_n \wedge
e)\vee (-e)$. Since  $u_n \in \ow{D}(Q) \cap L^\infty(X,m) \subseteq D(Q)_{\rm loc} \cap L^\infty(X,m)$ and $e$ has compact support, we obtain $u_n^e \in
D(Q)$ by Proposition~\ref{some properties}~(c). Using Proposition~\ref{some properties}~(d) we estimate
$$ Q(u_n^e)^{1/2} \leq Q(u_n\wedge e)^{1/2} + Q(e)^{1/2} \leq Q(u_n)^{1/2} + 2 Q(e)^{1/2}.$$
Therefore, $(u_n^e)_n$ is a bounded sequence in the Hilbert space
$(D(Q),\aV{\cdot}_Q)$. The Banach-Saks Theorem yields the existence
of a subsequence $(u_{n_k}^e)$ and a $v \in D(Q)$ such that $v_N =
\frac{1}{N}\sum_{k=1}^N u_{n_k}^ e$ is $\aV{\cdot}_Q$ convergent to
$v$. From  pointwise convergence of the $u_n$ and, since $e$ has
compact support, we infer $v_N \to (u \wedge e)\vee (-e)$ in
$L^2(X,m)$. Therefore, $(u \wedge e)\vee (-e) = v \in D(Q)$. Since
$(u \wedge e)\vee (-e) =  u $ on $G$, this shows $u \in D(Q)_{\rm
loc}$.

Let us turn to proving the  inequality: Since $X$ is locally compact
and separable, we find an increasing sequence  of relatively compact
open sets $G_l$ such that $\overline{G}_l \subseteq G_{l+1}$. By
regularity we can choose functions $e_l\in D(Q)\cap C_c(X)$, such
that $e_l = 1$ on $G_{l+1}$. Using the above and a diagonal sequence
argument, we may find a subsequence $(u_{n_k})$ such that for all
$l$ the sequence $v_N^l = \frac{1}{N}\sum_{k=1}^N u_{n_k}^{e_l}$
satisfies
$$\|v_N^l - u^{e_l}\|_Q \to 0, \text{ as }N\to \infty.$$
By \cite[Theorem 2.1.4]{FOT}, we infer that each sequence
$(\tilde{v}^l_N)_N$ has a q.e. convergent subsequence which
converges to $u^{e_l}$. By a diagonal sequence argument,
 we may assume that $\tilde{v}_N = \frac{1}{N}\sum_{k=1}^N
\tilde{u}_{n_k}$ is q.e. convergent towards $u$ (otherwise take a
subsequence). Since $J$ and $k$ charge no set of capacity zero,
Fatou's Lemma yields
\begin{align*}
\iint_{X \times X \setminus  {\it diag}} (\tilde u (x) - \tilde u
(y))^2&\, J(dx,dy)
     +\int_X  \tilde{u}^2 dk \\
     \leq & \liminf_{N\to \infty} \iint_{X \times X \setminus  {\it diag}}
(\tilde v_N (x) - \tilde v_N (y))^2\, J(dx,dy)
     +\int_X \tilde{v}_N^2 dk.
     \end{align*}
For the strongly local part, we obtain
\begin{align*}
\int_X d\muc(u) &= \lim_{l\to \infty} \int_{G_l} d\muc(u)\\
&= \lim_{l\to \infty} \int_{G_l} d\muc(u^{e_{l}})\\
&= \lim_{l\to \infty}  \lim_{N\to \infty} \int_{G_l} d\muc(v_N^l)\\
(\muc \text{ is local})\;\:\; &= \lim_{l\to \infty}  \lim_{N\to
\infty} \int_{G_l}
d\muc(v_N)\\
&\leq \liminf_{N\to \infty}  \int_{X} d\muc(v_N).
\end{align*}
The last inequality holds since for each $N$ the convergence in $l$ is
monotone (see Lemma~\ref{sequences}). Altogether  we obtain
\begin{align*}
Q(u)^{1/2} \leq \liminf_{N\to \infty} Q^{1/2}(v_N)
 \leq \liminf_{N\to \infty} \frac{1}{N}\sum_{k=1}^N Q(u_{n_k})^{1/2}
= \lim_{n\to \infty} Q(u_n)^{1/2},
\end{align*}
where the last step results from the assumption in the beginning of
the proof.

Step 2: For arbitrary $u\in L^{0}(X,m)$ as in the statement of the
theorem, the considerations of Step 1 applied to $u^{(k)}$ and the
sequence $(u_n^{(k)})_n$ show $u^{(k)} \in D(Q)_{\rm loc}$ for any
$k> 1$. Therefore, recalling the definition of $Q(u)$, we compute
\begin{align*}
\ow{Q}(u) &= \lim_{k\to\infty} Q(u^{(k)})\\
\text{(Step 1)} &\leq \liminf_{k\to\infty} \liminf_{n\to \infty} Q(u^{(k)}_n)\\
&\leq \liminf_{n\to\infty} \liminf_{k\to \infty} Q(u^{(k)}_n)\\
&\leq \liminf_{n\to\infty}  \ow{Q}(u_n).
\end{align*}
Here, we used in the third step that for each $n$ the convergence in
$k$ is monotone (Proposition~\ref{some properties}~(b)) and Lemma
\ref{sequences}. This finishes the proof.
\end{proof}

Recall that a functional $q$ on some real linear space $F$ is called a \emph{quadratic form}, if 
$$q(f+g) + q(f-g) = 2q(f) + 2q(g)\text{ and } q(af) = a^2q(f)$$
for any $f,g \in F$ and $a \in \RR$. Any quadratic form $q$ induces a bilinear form via polarization which we also denote by $q$. The following theorem shows that we can apply this concept to $F = \ow{D}(Q)$ and $q = \ow{Q}$. 

\begin{theorem} \label{thm:Quadratic form and cut-off}
The map $\ow{Q}:\ow{D}(Q) \to [0,\infty)$ is a quadratic form. Furthermore, for any normal contraction $C: \RR \to\RR$ and any $u \in \ow{D}(Q)$ we have $C\circ u \in \ow{D}(Q)$ and
$$\ow{Q}(C\circ u) \leq \ow{Q}(u).$$
\end{theorem}

\begin{proof}
We first show the contraction property. Let $C:\RR \to \RR$ be a normal contraction. Now, Fatou's Lemma for $\ow{Q}$ and Proposition~\ref{some properties}~(a) yields
$$\ow{Q}(C\circ u) \leq \liminf_{n \to \infty} \ow{Q}(C\circ u^{(n)}) = \liminf_{n \to \infty} Q(C\circ u^{(n)}) \leq \liminf_{n \to \infty} Q(u^{(n)}) =\ow{Q}(u).$$
It remains to show that $\ow{Q}$ is a quadratic form. Let $a \in \RR$ and $u \in \ow{D}(Q)$ be given. Fatou's Lemma for $\ow{Q}$ and the fact that $Q$ is a quadratic form on $\{u \in D(Q)_{\rm loc} \, : \, Q(u) < \infty\}$ yields
$$\ow{Q}(u) \leq \liminf_{n\to \infty} \ow{Q}(\frac{1}{a} (au)^{(n)}) = \liminf_{n\to \infty} Q(\frac{1}{a} (au)^{(n)}) = \frac{1}{a^2} \ow{Q}(au).$$
For the inequality $\ow{Q}(u) \geq \frac{1}{a^2} \ow{Q}(au)$  we note, that for each $n$ the map $x \mapsto \frac{1}{a}(ax)^{(n)}$ is a normal contraction and compute
$$\frac{1}{a^2} \ow{Q}(au) = \lim_{n\to \infty}\frac{1}{a^2}Q( (au)^{(n)}) = \lim_{n\to \infty}Q(\frac{1}{a} (au)^{(n)}) = \lim_{n\to \infty}\ow{Q}(\frac{1}{a} (au)^{(n)}) \leq \ow{Q}(u).$$
Now, let $u,v \in \ow{D}(Q)$ be given.  Fatou's Lemma for $\ow{Q}$ and the fact that $Q$ is a quadratic form on $\{u \in D(Q)_{\rm loc} \, : \, Q(u) < \infty\}$ yields
\begin{align*}
 \ow{Q}(u+v) +  \ow{Q}(u-v) &\leq \liminf_{n\to \infty} (Q(u^{(n)} + v^{(n)}) +  (Q(u^{(n)} - v^{(n)})) \\
 &= \liminf_{n\to \infty} (2Q(u^{(n)}) +  2 Q(v^{(n)}))\\
 &= 2\ow{Q}(u) + 2 \ow{Q}(v).
\end{align*}
Since the above inequality is true for arbitrary functions, we can apply it to $u' = u + v$ and  $v' = u - v$ to obtain
$$4\ow{Q}(u) + 4 \ow{Q}(v) = \ow{Q}(u' + v') + \ow{Q}(u'-v') \leq 2 \ow{Q}(u') + 2 \ow{Q}(v') = 2 \ow{Q}(u+v) + 2 \ow{Q}(u-v).$$
The first equality is a consequence of $\ow{Q}(aw) = a^2 \ow{Q}(w)$ which was proven above. This finishes the proof.

\end{proof}

\begin{remark}
\begin{itemize}
  \item In general, $\widetilde{D}(Q)$ does not need to be included in
$L^2(X,m)$ and, hence,  $(\ow{Q},\widetilde{D}(Q))$ is not a Dirichlet form in the usual sense. See
\cite{Kuw3, CF} for the background of this definition and properties
in the quasi-regular case.
 \item The importance of $\ow{D}(Q)$ stems from the fact that $\ow{Q}$ is a finite quadratic form on this space inducing a bilinear form by
polarization (which will also be called $\ow{Q}$).  It will follow from the previous theorems  that this form does not only extend $(Q,D(Q))$ but also
provides an extension of $(Q,D(Q)_e)$. Furthermore, it yields the well known fact that $(\ow{Q},\ow{D}(Q)\cap~L^2(X,m))$ is a closed form (see, e.g., \cite{Kuw3} and \cite{Che}).
\item The above lower semi-continuity of $Q$ on its reflected Dirichlet space with respect to pointwise convergence seems to be new. As $D(Q)_e \subseteq \ow{D}(Q)$ (see below), we obtain an extension of \cite[Corollary 1.1.9]{CF} for regular Dirichlet forms.
\end{itemize}
\end{remark}

We now come to a crucial definition for the subsequent
considerations. In the previous sections we used various operators
to characterize the investigated properties. In the regular setting
we will see below that all these operators are restrictions of
$$\cL: D(\cL)\to L^1_{\rm loc}(X,m),$$
where
$$ D(\cL) = \{u \in \widetilde{D}(Q)\, | \, \exists f \in L^1_{\rm loc}(X,m)\;\forall
v \in D(Q) \cap C_c(X)\; \ow{Q}(u,v) = \as{f,v} \}$$ on
which it acts by
$$\cL u = f.$$ \label{Uber Laplace}

The following proposition is clear from the definitions.
\begin{prop}\label{Domain L regular} The operator $\cL$ is an extension of $L$. The domain of $L$ satisfies
$$D(L) = \{u \in D(Q) \cap D(\cL)\, \mid \, \cL u \in L^2(X,m)\}.$$
\end{prop}
\begin{remark}
In a certain sense $\cL$ is a distributional extension of $L$. In
many situations its domain and action are known explicitly, see
Sections \ref{section;Graph}, \ref{section;Manifold} and
\ref{section;MG}.
\end{remark}

%%%%%%%%%%%%%%%%%%%%%%%%%%%%%%%%%%%%%%%%%%%%%%%%%%%%%%%%%%%%%%%%%%%%%%%%%%%%

\subsection{Extensions of Dirichlet forms: The uniqueness of Silverstein\rq{}s
extension}
\label{subsection;Silverstein}
We will now apply the theory of Section \ref{section;Extensions} to the form $Q$
on $\mathcal{D} = D(Q)$ and to $\ow{Q}$ on  $\Ds = D(Q^{\max}):=\ow{D}(Q) \cap L^2(X,m)$. We
think of $D(Q)$ as encoding ``Dirichlet boundary conditions'' and of $D(Q^{\max})$
as encoding ``Neumann type boundary conditions''. We write $Q^{\max}$ whenever we
refer to $\ow{Q}$ on $D(Q^{\max})$ and denote the associated positive operator by
$L^{\max}$. The following propositions assure that the theory of Section~\ref{section;Extensions} can be applied.

\begin{prop}[$Q^{\max}$ as Dirichlet form] The form
$Q^{\max}$ is a Dirichlet form. The space $D(Q^{\max}) \cap L^\infty(X,m)$ is given by
those $u \in D(Q)_{\rm loc} \cap L^2(X,m) \cap L^\infty(X,m)$ satisfying
$$\int_X
d\mu^{(c)}(u) +  \iint_{X \times X \setminus {\it diag}} (\tilde u
(x) - \tilde u (y))^2\, J(dx,dy)
     +\int_X \tilde u^2 dk < \infty.$$
\end{prop}
\begin{proof} The closedness of $Q^{\max}$ follows from Fatou's lemma, Theorem \ref{Fatou} while the fact that $Q^{\max}$ is a Markovian quadratic form is a consequence of Theorem \ref{thm:Quadratic form and cut-off}. The fact about the action of $Q^{\max}$ on bounded function is a consequence of the fact that $\ow{Q}$ and the extension of $Q$ to $D(Q)_{\rm loc}$ agree on $L^\infty(X,m)$. This finishes the proof.
\end{proof}

\begin{prop} \label{MP regular} The forms $(Q,D(Q))$ and $(Q^{\max},D(Q^{\max}))$
satisfy the maximum principle (MP), i.e.,
 positive $f \in L^2(X,m)$ satisfy the inequality
$$(L+\alpha)^{-1}f \leq (L^{\max}+\alpha)^{-1}f.$$

\end{prop}
\begin{proof} We use Theorem \ref{max principle} with $\mathcal{D} = D(Q)$ and
$\Ds = D(Q^{\max})$. Choose $(G_n)$ to be an increasing sequence of open,
relatively compact sets such that $\overline{G_n} \subseteq G_{n+1}$ and $\cup G_n = X$. Then, the inclusion
$$C_c(X) \cap D(Q) \subseteq \mathcal{C} := \bigcup_{n \geq 1} D(Q^{\max}_{G_n})$$
holds, where $D(Q^{\max}_{G_n}) = \{u \in D(Q^{\max}) \mid u = 0\; m\text{-a.e. on } X \setminus G_n\}.$  Furthermore, we have $\mathcal{C} \subseteq D(Q)$. To see this, let $u \in D(Q^{\max}_{G_n})$ be given. Without loss of generality, we may assume $u \in L^\infty(X,m)$. Since $D(Q^{\max}) \cap L^\infty(X,m) \subseteq D(Q)_{\rm loc}$, there exists a function $v \in D(Q) \cap L^\infty(X,m)$ such that $v = u$ on $G_{n+1}$. By the regularity of $Q$ there exists a function $\varphi \in D(Q)\cap L^\infty(X,m)$ such that $\varphi = 1$ on $G_n$ and $\varphi  = 0$ on $G_{n+1}$. Since $D(Q)\cap L^\infty(X,m)$ is an algbra, c.f. \cite[Theorem 1.4.2]{FOT}, we obtain $u = u \varphi = v \varphi \in D(Q)$. Now, Theorem \ref{max principle} shows  the statement as $C_c(X) \cap D(Q)$ is dense in $D(Q)$.
\end{proof}

\begin{remark}
\begin{itemize} \item In the literature $(Q^{\max},D(Q^{\max}))$ is called the active
reflected Dirichlet space of $(Q,D(Q))$. This terminology stems from
the following observation. If $(Q,D(Q))$ is the standard Dirichlet
energy on an open subdomain of $\mathbb{R}^n$ considered  with
Dirichlet boundary conditions, then $(Q^{\max},D(Q^{\rm max}))$ is
the Dirichlet energy with Neumann boundary conditions. In terms of
stochastic processes this means that $(Q,D(Q))$ corresponds to
Brownian motion which is killed upon hitting the boundary while
$(Q^{\max},D(Q^{\rm max}))$ corresponds to Brownian motions which is
reflected at the boundary.

\item We use the notation $Q^{\max}$ as it is the maximal Silverstein
extension of $Q$ (see e.g. \cite{Kuw3, CF}). Recall that an
extension $\hat Q$ of $Q$ is called a \emph{Silverstein extension}
if $u \cdot v \in D(Q)$ for all $u \in D(Q) \cap L^\infty (X)$, $v
\in D(\hat Q) \cap L^\infty (X)$.
\end{itemize}
\end{remark}

Recall the definition of the operator $L'$ in Section \ref{section;Extensions}.

\begin{prop}
If $\mathcal{D} = D(Q)$ and $\Ds = D(Q^{\rm max})$, then the domain of $L'$ satisfies
$$D(L') = \{u \in D(\cL )\cap L^2(X,m)\, |\, \cL u \in L^2(X,m)\}.$$
\end{prop}
\begin{proof} This follows from the definitions and the regularity of $Q$.
\end{proof}

Therefore, our main statement of Section \ref{section;Extensions},
Theorem~\ref{thm;2.2},  has  now an immediate consequence.
\begin{theorem} \label{H neq H0}
Let $Q$ be a regular Dirichlet form. Then the following assertions are
equivalent.\begin{itemize}
\item[(i)] $Q \neq Q^{\max}$.
\item[(ii)] There exists a nontrivial $u \in D(\cL)\cap L^1(X,m)\cap L^2(X,m)$
such that $\cL u \in L^2(X,m)$ with $u\geq 0$, $\cL u \leq 0$ and $ \cL u \neq
0$.
\end{itemize}
If the assertions hold, then  the function $u$ in $\mbox{\rm (ii)}$ can additionally be
chosen to be bounded.
\end{theorem}
We can now  also give a characterization of $Q = Q^{\max}$ in terms
of Green's formula.
\begin{theorem}[Characterization of $Q = Q^{\max}$ for regular forms via Greens formula]
 \label{thm;US1}
Let $Q$ be a regular Dirichlet form and $L$ its associated operator. Then the
following assertions are equivalent.
\begin{itemize}
\item[(i)] $Q = Q^{\max}$.
\item[(ii)] $D(L) = \{u \in D(\cL )\cap L^2(X,m)\, | \, \cL u \in L^2(X,m)\}$.
\item[(iii)] For all $u \in D(\cL)\cap L^2(X,m)$ such that $\cL u \in L^2(X,m)$
and all $v \in \ow{D}(Q)\cap L^2(X,m)$ the following equality holds
$$\ow{Q} (u,v) = \int_X \cL uv\, dm.$$
\end{itemize}
\end{theorem}
\begin{proof}
(i) $\Longrightarrow$ (ii): The domain of $L$ is given by
\begin{gather*}
D(L) = \{u\in D(Q)\mid \mbox{there exists }w\in L^2(X,m) \text{ s.t. } Q(u,v)
= \as{w,v} \text{ for all }v \in D(Q)\}. \label{D(L)}
\end{gather*}
Since $Q$ is regular, this implies $D(L) \subseteq \{u \in D(\cL
)\cap L^2(X,m)\mid \cL u \in L^2(X,m)\}$.

Now let $u \in D(\cL)\cap L^2(X,m)$ with $\cL u \in L^2(X,m)$ be
given. Then, the definition of $\cL$ and of $Q^{\max}$ together with
the equality $Q =Q^{\max}$ imply
$$\as{\cL u,v} = Q^{\max}  (u,v) = Q (u,v)$$
for all $v \in D(Q)$. This directly gives $u \in D(L)$ (and $L u =
\cL u$).

(ii) $\Longrightarrow$ (iii):
%Using the definition  of ${D(L)}$ with $Q^{\max}$ instead of $Q$,
The definition of $L^{\max}$ shows  $D(L^{\max}) \subseteq \{u \in
D(\cL)\cap L^2(X,m): \cL u \in L^2(X,m)\}$ and $\cL u = L^{\max} u$
for $u\in D (L^{\max})$.  Then (ii) gives that $L^{\max}$ is a
restriction of $L$. As both $L$ and $L^{\max}$ are selfadjoint, we
infer $L =L^{\max}$ and this easily  yields (iii).

(iii) $\Longrightarrow$ (i): Assume $Q \neq Q^{\max}$. Then, by Theorem~\ref{H neq H0},
there exists    $u \in D(\cL)\cap L^2(X,m)$  such that $\cL u \in L^2(X,m)$
with $u\geq 0$, $\cL u \leq 0$ and $ \cL u \neq 0$. This $u$ satisfies
$Q(u,u)\geq 0$ and
$$\int_X \cL uu\, dm < 0$$
which contradicts (iii).\end{proof}

\subsection{Stochastic completeness}
\label{subsection;SC} In the regular setting we can  give a more
explicit characterization of stochastic completeness since we can
compute the domain of $L$. For this the following proposition is
needed.

\begin{prop}\label{SCimpliesREG}
Let a regular Dirichlet form $Q$ be stochastically complete. Then $Q
= Q^{\max}$.
\end{prop}
\begin{proof} We show $(L+1)^{-1} = (L^{\max}+1)^{-1}$. Let $(e_n)$ be a
sequence in $D(Q)\cap C_c(X)$ such that $0\leq e_n \leq 1$ and $e_n
\uparrow 1$ $m$-almost everywhere. By the maximum principle (MP),
Proposition~\ref{MP regular}, and stochastic completeness, we obtain
$$1 = (L+1)^{-1}1 = \lim_{n\to \infty} (L+1)^{-1}e_n  \leq \lim_{n\to \infty}
(L^{\max}+1)^{-1}e_n = (L^{\max}+1)^{-1}1\leq 1.$$
This shows $(L^{\max}+1)^{-1}1 = 1$. Now, let $0 \leq f \in L^1(X,m) \cap
L^2(X,m)$ be given. Since  $(L^{\max}+1)^{-1}f - (L+1)^{-1}f \in
L^1(X,m)\cap L^2(X,m),$ we obtain by Lebesgue's theorem
\begin{align*}
0&\leq  \as{(L^{\max}+1)^{-1}f - (L+1)^{-1}f,1} \\
&\leq  \lim_{n\to \infty} \as{(L^{\max}+1)^{-1}f - (L+1)^{-1}f, e_n} \\
&= \lim_{n\to \infty} \as{f, (L^{\max}+1)^{-1}e_n - (L+1)^{-1}e_n}\\
&= 0.
\end{align*}
This shows $(L+1)^{-1} = (L^{\max}+1)^{-1}$ and our claim follows.
\end{proof}

\begin{remark}
The previous proposition is known see, e.g.,
\cite[Theorem~6.3]{Kuw3}. Note however, that our proof only uses the
estimate $(L+1)^{-1}f \leq (L^{\max}+1)^{-1}f$ and did not rely on
the regularity of $Q$. Thus, it  also holds for pairs of forms
satisfying the maximum principle (MP) (c.f.
Section~\ref{section;Extensions}).
\end{remark}

With this at hand our main theorem on stochastic completeness reads as follows.
\begin{theorem}[Characterization stochastic completeness for regular forms] \label{SC-Reg}
Let $Q$ be a regular Dirichlet form. Then the following assertions are
equivalent.
\begin{itemize}
\item[(i)] $Q$ is stochastically complete.
\item[(ii)] For all $u \in D(\cL)\cap L^1(X,m) \cap L^2(X,m)$ such that $\cL u
\in L^1(X,m) \cap L^2(X,m)$ the following equality holds
$$\int_X \cL u \, dm = 0.$$
\end{itemize}
\end{theorem}
\begin{proof}
(i) $\Longrightarrow$ (ii): Since $Q$ is stochastically complete, $Q = Q^{\max}$
holds by the previous proposition. Then, Theorem~\ref{thm;US1} shows $D(L) =
\{D(\cL) \cap L^2(X,m)\mid\cL u \in L^2(X,m) \}$. Therefore, Theorem~\ref{thm;SC} implies the statement.

(ii) $\Longrightarrow$ (i): This is an immediate consequence of Theorem~\ref{thm;SC} since $D(L) \subseteq  \{D(\cL) \cap L^2(X,m)\mid\cL u \in
L^2(X,m) \}$ is  satisfied by Proposition~\ref{Domain L regular}.
\end{proof}

%%%%%%%%%%%%%%%%%%%%%%%%%%%%%%%%%%%%%
\subsection{Recurrence}\label{subsection;recurrence}

We  improve the results on recurrence in the regular setting. For this we
need  that the action of $\ow{Q}$ on the space of functions of finite energy is compatible with the action of $Q$ on its extended Dirichlet space.

\begin{lemma} The inclusion $D(Q)_e \subseteq \widetilde{D}(Q)$ holds and the
extension of $Q$ to  $D(Q)_e$ equals $\ow{Q}$ on
$D(Q)_e$. Furthermore, if $Q$ is recurrent,  then $D(Q)_e = \widetilde{D}(Q)$.
\end{lemma}
\begin{proof}
Recall the definition of the extension of $Q$ to $D(Q)_e$ in the
beginning of Section~\ref{section;recurrence}.  Let $u \in D(Q)_e$ be given  and let $(u_n)\subseteq D(Q)$ be an approximating
sequence for $u$. Then, Theorem~\ref{Fatou} shows $u \in
\widetilde{D}(Q)$ and
$$|\ow{Q}(u)^{1/2}-\ow{Q}(u_n)^{1/2}| \leq \ow{Q}(u-u_n)^{1/2} \leq \liminf_{m\to
\infty}\ow{Q}(u_m-u_n)^{1/2} = \liminf_{m\to
\infty}Q(u_m-u_n)^{1/2}.$$
This shows $Q(u_n) \to \ow{Q}(u)$ which was the first claim.

Now, assume $Q$ is recurrent.  Then there exists a sequence $e_n \in D(Q)\cap
C_c(X)$ such that $0\leq e_n \leq 1$, $e_n \to 1$ $m$-a.e.\ and $Q(e_n) \to 0$
(see Appendix for details).  Let $u \in \widetilde{D}(Q)$ be given. Without
loss of generality we may assume that $0 \leq u \leq 1$ (else approximate,
rescale, split in positive and negative part). The function $u \wedge e_n$ has compact support and it follows from the definition of
$\widetilde{D}(Q)$ that it belongs to $ D(Q)_{\rm loc} \cap L^\infty(X,m)$. Therefore, Proposition~\ref{some properties}~(c) shows $u \wedge e_n \in D(Q).$   Since Proposition~\ref{some properties}~(d) implies
$$Q(u \wedge e_n)^{1/2} \leq Q(u)^{1/2}  + Q(e_n)^{1/2},$$
the sequence $(u\wedge e_n)$ is bounded with respect to the inner
product space $(Q,D(Q))$. Hence, by some version of the
Banach-Saks Theorem, it has a subsequence $u\wedge e_{n_k}$ such
that  $v_N = \frac{1}{N}\sum_{k=1}^N u\wedge e_{n_k}$ is a
$Q$-Cauchy sequence. Since $u\wedge e_n \to u$ $m$-a.e., we also
obtain $v_N \to u$ $m$-a.e., and, therefore, $u \in D(Q)_e$.
\end{proof}

With this at hand our main result on recurrence reads as follows.
\begin{theorem}[Characterization recurrence for regular forms]  \label{recurrence,regular}
Let $Q$ be an irreducible regular Dirichlet form. Then the following
assertions are equivalent.
\begin{itemize}
\item[(i)] $Q$ is recurrent.
\item[(ii)] For all $u \in D(\cL)$ such that $\cL u \in L^1(X,m)$ the following equality holds
$$\int_X \cL u\, dm = 0 .$$

\item[(iii)] The killing measure $k$ vanishes and for all $u\in D(\cL )$, such
that $\cL u \in L^1(X,m)$ and for all $v \in \widetilde{D}(Q) \cap
L^\infty(X,m)$ the following equality holds
$$\ow{Q}(u,v) = \int_X \cL uv \,dm. $$

\end{itemize}
\end{theorem}
\begin{proof}
(ii) $\Longrightarrow$ (i): This immediately follows from Theorem~\ref{thm;recurrence} and the previous lemma.

(i) $\Longrightarrow$ (iii): The form $Q$ is recurrent, therefore, the killing measure $k$
vanishes. Let $v \in \ow{D}(Q)\cap L^\infty(X,m)$ be given. Since $Q$ is recurrent,
the previous Lemma shows $v \in D(Q)_e$, hence $v$ possesses an approximating
sequence $(v_n)$ in  $D(Q)$. By regularity we may choose $(v_n)$ in  $D(Q)\cap
C_c(X)$. Furthermore, we may assume that the $v_n$ are uniformly bounded by
$\|v\|_\infty$. Using Lebesgue's  theorem, we then obtain for $u\in D(\cL)$ with $\cL u  \in L^1$
\begin{align*}
\ow{Q}(u,v)  = \lim_{n\to \infty}\ow{Q}(u,v_n) = \lim_{n\to \infty}\int_X \cL uv_n \,dm= \int_X \cL uv \,dm.
\end{align*}
(iii) $\Longrightarrow$ (ii): Assume (ii) does not hold. Then, there exists $u\in D(\cL)$ with $\cL u \in L^1$ such that
$$\int_X \cL u\, dm \neq  0.$$
Since $k$ vanishes, we obtain $1 \in \ow{D}(Q)$ and $\ow{Q} (1,u) = 0$. This contradicts (iii).
\end{proof}

\subsection{The relation between the concepts} As an application of the above
criteria we finish this section by discussing the relation of the various
concepts.

\begin{theorem}[Relation between the concepts]  \label{regular,connection}
Suppose $Q$ is a regular  irreducible Dirichlet form. Consider the
following statements.
\begin{itemize}
\item[(i)] $Q$ is recurrent.
\item[(ii)] $Q$ is stochastically complete.
\item[(iii)] $Q = Q^{\max}$.
\end{itemize}
Then, the implications $\mathrm{(i)}$ $\Longrightarrow$
$\mathrm{(ii)}$ and $\mathrm{(ii)}$ $\Longrightarrow$
$\mathrm{(iii)}$ are always true. If the killing measure $k$
vanishes and $m(X) < \infty$ the above assertions are equivalent.
\end{theorem}

\begin{proof}
The implication  (i) $\Longrightarrow$ (ii) immediately follows from Theorem~\ref{recurrence,regular} and Theorem~\ref{SC-Reg}. The implication (ii) $\Longrightarrow$ (iii) is
the statement of Proposition~\ref{SCimpliesREG}. Suppose now $m(X) < \infty$ and
$k = 0$. It remains to show the implication (iii) $\Longrightarrow$ (i). This follows from the fact
that $m(X) < \infty $ implies $1 \in D(Q^{\rm max})$ and $k = 0$ implies $Q^{\rm max}(1) = 0$.
\end{proof}

\begin{remark}
The implications (i) $\Longrightarrow$ (ii) and (ii) $\Longrightarrow$ (iii) are
well-known. However, the statement on the equivalence in case of finite measure seems to be interesting. See the proof of Theorem~\ref{parabolicgraphvar} for
an application.
\end{remark}

%%%%%%%%%%%%%%%%%%%%%%%%%%%%%%%%%%%%%%%%%%%%%%%%%%%%%%%%%%%%%%%%%%%%%%%%%%%%%%%%
%%

\section{Application to graphs}\label{section;Graph}
In this section we apply the results obtained above to graphs. Here, we use the
framework of regular Dirichlet forms on graphs discussed in various recent
works.  After a discussion of the background, we specify the space of the functions of finite energy. Then, we turn to extensions, stochastic
completeness and to recurrence,

The salient feature of our discussion is that  the operator $\cL $
associated to a regular Dirichlet form in the previous section is known
explicitly in the graph case. This makes it possible to present unified formulations of the results.

Unbounded Laplacians on graphs have become a focus of research  in
various recent works, see e.g., \cite{CdVTHT,Hua,KL,HKLW,Woj,Woj2},
and references therein.

The subsequent discussion of the setting essentially follows \cite{KL} (see
\cite{HK,HKLW} as well) to which we refer for further details and proofs.   Let
$V$ be a countable set and $C(V)$ be the set of all real-valued functions on $V$. For a measure $m:V\to(0,\infty)$ let
$\ell^2(V,m)=\{u:V\to\RR\mid \sum_{x\in V}|u(x)|^2m(x)<\infty\}$
and denote the corresponding scalar product by $\as{\cdot,\cdot }$ and the
corresponding norm by $\aV{\cdot}$.

Let $b:V\times V\to[0,\infty)$ be symmetric with zero diagonal and
assume $\sum_{y\in V} b(x,y)<\infty$ for all $x\in V$. Furthermore,
let $c:V\to[0,\infty)$. We then call  $(b,c)$  a \textit{weighted
graph over $V$} and refer to $V$ as  the \textit{vertex set}.
Moreover,   $x,y\in V$ are \textit{connected by an edge  with weight
$b(x,y)$} whenever $b(x,y)>0$. In this case, we write $x\sim y$.
Furthermore, $c$ encodes one-way-edges from $x$ whenever $c(x)>0$.

We say a set $W\subseteq V$ is \emph{connected} if for all $x,y\in W$ there exists a finite sequence of vertices  $x=x_0,\ldots,x_n=y$ in $W$ such
that $x_j\sim x_{j+1}$, $j=0,\ldots,n-1$. We call such a sequence of vertices a
\emph{path} from $x$ to $y$.

Let $\ow{Q}_{b,c}:C(V) \to[0,\infty]$ be given by
$$\ow{Q}_{b,c}(u)=\frac{1}{2}\sum_{x,y\in V} b(x,y)(u(x)-u(y))^2+\sum_{x\in V}
c(x)u(x)^2.$$
We are interested in the space
$$\widetilde{D}=\{u\in C(V)\mid \ow{Q}_{b,c}(u)<\infty\}.$$
By the summability assumption on $b$ the inclusion $C_c(V)\subseteq
\widetilde{D}$ follows easily, where  $C_c(V)$ is the space of
finitely supported functions. By polarization $\ow{Q}_{b,c}$
extends to a symmetric bilinear form on $\widetilde{D}\times
\widetilde{D}$. This bilinear form will again be denoted by
$\ow{Q}_{b,c}$.

There is a regular Dirichlet form associated with $(b,c)$ introduced
next. Let $Q$ be the restriction of $\ow{Q}_{b,c}$ to
$$D(Q)=\ov{C_c(V)}^{\aV{\cdot}_{Q}},$$
where $\aV{\cdot}_{Q}^2=\aV{\cdot}^2+\ow{Q}_{b,c}(\cdot).$

By Fatou's lemma $\ow{Q}_{b,c}$ is lower semi-continuous and, hence, every restriction is closable. Thus, the form $Q$ is closed by definition of $D(Q)$. Moreover,  $C_c(V)\subseteq D(Q)$ implies that $Q$ is regular, i.e., $D(Q)\cap
C_c(V)$ is dense in $C_c(V)$  with respect to the supremum norm $\aV{\cdot}_{\infty}$ and $D(Q)$ with respect to $\aV{\cdot}_{Q}$. One can check that $(Q,D(Q))$ is a Dirichlet form (see
\cite[Theorem~3.1.1]{FOT}),
which we  call the {\em regular Dirichlet form associated to $(b,c)$}.

\subsection{Functions of finite energy}
In view of the theory developed in Chapter~\ref{section;Regular}, we
aim at  determining  the associated space of functions of finite
energy $\ow{D}(Q)$ and the maximal form $D(Q^{\rm max})$.

\begin{prop} \label{functions of finite energy discrete graph}
Let $Q$ be the regular Dirichlet form on $\ell^2(V,m)$ associated to
the graph $(b,c)$. Then, $\ow{D}(Q) = \ow{D}$ and $\ow{Q}$ on $\ow{D}(Q)$ is given by $\ow{Q}_{b,c}$. Furthermore, $Q^{\rm max}$
is the  restriction of $\ow{Q}_{b,c}$ to the domain $D(Q^{\rm max}) =
\ow{D} \cap \ell^2(V,m)$.
\end{prop}
\begin{proof}
The equality $\ow{D}(Q) = \ow{D}$ follows from   $D(Q)_{\rm loc} =
C(V)$ and the monotone convergence theorem. The rest is clear from
the definitions.
\end{proof}

Let us now turn to the associated operators. We  let
\begin{equation*}\label{ftilde}
\widetilde{F}:=\{ w : V\to \RR \mid \sum_{y\in V}
b(x,y)|w(y)|<\infty\;\mbox{for all  $x\in V$ } \}
\end{equation*}
and define $\widetilde{L} : \widetilde{F}\longrightarrow  C(V)$ via
$$\widetilde{L} w (x) :=\frac{1}{m(x)} \sum_{y\in V} b(x,y) (w(x) - w(y)) +
\frac{c(x)}{m(x)}  w(x). $$ Here, indeed the sum exists
 for each $x\in V$ due to  $w\in \widetilde{F}$. Then $L$, the associated operator of $Q$, is a restriction of
$\widetilde{L}$ with domain satisfying
$$D(L) \subset \{u\in \ell^2(V,m) \mid \widetilde{L}u \in \ell^2(V,m)\}.$$
As mentioned above, details can be found in \cite{KL}.
Here, we just briefly discuss the crucial link between the operator
$\widetilde{L}$ and the form $Q$.
This link is given by the following Green-type-formula.   Various variants
can be found in \cite{HK,KL,HKLW}.

\begin{lemma}[Green type formula] \label{Green}
(a)  The set $\widetilde{D}$ is contained in  $\widetilde{F}$.

(b) For all $w\in \widetilde{F}$ and $v\in C_c (V)$, the following equality holds
$$\widetilde{Q}(w,v) = \sum_x (\widetilde{L}  w) (x)  v(x) m(x) = \sum_x w (x)
(\widetilde{L} v) (x).$$
\end{lemma}
\begin{proof} Statement (a) is part of \cite[Proposition~2.8]{HKLW} and statement (b) is contained in \cite[Lemma~4.7]{HK}.
\end{proof}

With this at hand we can identify the operator $\mathcal{L}$ for the
regular Dirichlet form associated to the graph $(b,c)$. Recall its
definition on Page \pageref{Uber Laplace}.

\begin{theorem}\label{prop;Graphs,L}
Let $Q$ be the regular Dirichlet form associated with the graph $(b,c)$. Then, $\cL$ is the restriction of $\ow{L}$ to the domain $D(\cL) = \ow{D}$.
\end{theorem}

\begin{proof}
This is a consequence of Lemma \ref{functions of finite
energy discrete graph}, the previous lemma and the definition of
$\cL$.
\end{proof}

%%%%%%%%%%%%%%%%%%%%%%%%%%%%%%%%%%%%%%%%%%%%%%%%%%%%%%%%%%%%%%%%%%%%%%%%%%%%%%%%
%
\subsection{Extensions of Dirichlet forms} \label{section;Graphs;extensions}

The next result is a direct  application of Theorem~\ref{H neq H0}.

\begin{theorem}\label{thm;graph;H neq H0}
Let $Q$ be the regular Dirichlet form associated to $(b,c)$. Then the following assertions are
equivalent.
\begin{itemize}
\item[(i)] $Q \neq Q^{\max}$.
\item[(ii)] There exists a nontrivial $u \in D(Q^{\max})\cap \ell^1(V,m)\cap \ell^{\infty}(V)$
such that $\ow L u \in \ell^2(V,m)$ with $u\geq 0$, $ \ow Lu \leq 0$ and $ \ow L u \neq
0$.
\end{itemize}
\end{theorem}
\begin{remark}
The previous theorem is a slight extension of \cite[Corollary
4.3]{HKLW}. Specifically, it suffices to look for subsolutions in
$\ell^{1}$ for the direction (ii) $\Longrightarrow$ (i) is not found
in \cite{HKLW}.
\end{remark}

Next, we come to the application of Theorem~\ref{thm;US1}.

\begin{theorem} \label{thm;Graph;US1}
Let $Q$ be the regular Dirichlet form associated to $(b,c)$. Then, the
following assertions are equivalent.
\begin{itemize}
\item[(i)] $Q = Q^{\max}$.
\item[(ii)] $D(L) = \{u \in  D(Q^{\max} )\mid \ow L u \in \ell^2(V,m)\}$.
\item[(iii)] For all $u \in  D(Q^{\max} )$ such that $\ow L u \in \ell^2(V,m)$
and all $v \in D(Q^{\max} )$,
$$\ow{Q}(u,v) = \sum_{x\in V} \ow L u(x)v(x)m(x).$$
\end{itemize}
\end{theorem}

% Finally, the application of Theorem~\ref{thm;US2} yields an extension of \cite[Corollary~1]{HKMW} to non locally finite graphs.
%
% \begin{theorem}\label{thm;graph;US2}
% Let $Q$ be the regular Dirichlet form associated to $(b,c)$. Suppose there is an intrinsic metric such that the distance balls are finite. Then,  $Q=Q^{\rm max}$.
% \end{theorem}

%%%%%%%%%%%%%%%%%%%%%%%%%%%%%%%%%%%%%%%%%%%%%%%%%%%%%%%%%%%%%%%%%%%%%%%%%%%%%%%%

\subsection{Stochastic Completeness} \label{Stochastic}

We now come to the application of Theorem \ref{SC-Reg} concerning stochastic completeness.
It is known (see \cite{KL}) that the generator $L^{(1)}$ of the $\ell^1$
semigroup is a restriction of $\widetilde{L}$
to the domain $ D(L^{(1)})$  which satisfies
$$ D(L^{(1)}) \subseteq \{ u \in \ell^1(V,m)\mid \widetilde{L}u \in
\ell^1(V,m)\}.$$
Here, equality holds if furthermore the following assumption (A) is
satisfied (compare \cite[Theorem 5]{KL}):
 \begin{itemize}  \item [(A)] Every infinite path $(x_n)$ of vertices has infinite measure, i.e., $\sum_n  m(x_n)=\infty$. \end{itemize}
With this the following theorem  is then an immediate consequence of
Theorem \ref{SC-Reg}
 and Theorem~\ref{prop;Graphs,L}.

% The implication (iii)$\Longrightarrow$ (i) follows as, clearly, (iii) implies (ii)
% and the implication (i)$\Longrightarrow$ (iii)  under

%
\begin{theorem}
Let $Q$ be the regular Dirichlet form associated to $(b,c)$. Consider the statements:
\begin{itemize}
\item[(i)] $Q$ is stochastically complete.
\item[(ii)]
For all $u \in \ow{D} \cap \ell^2(V,m) \cap \ell^1(V,m)$ satisfying
$\widetilde{L}u \in \ell^1(V,m) \cap \ell^2(V,m)$ we have
$$ \sum_{x\in V} \widetilde{L}u (x)m(x) = 0.$$

\item [(iii)]
For all $u \in \ell^1(V,m)$ satisfying $\widetilde{L}u \in
\ell^1(V,m)$ we have
$$ \sum_{x\in V} \widetilde{L}u (x)m(x) = 0.$$
\end{itemize}
Then,  $\mathrm{(i)}$ $\Longleftrightarrow$ $\mathrm{(ii)}$ and  $\mathrm{(iii)}$ $\Longrightarrow$ $\mathrm{(i)}$. If (A) is satisfied, then  $\mathrm{(i)}$ $\Longleftrightarrow$ $\mathrm{(iii)}$.
\end{theorem}

\begin{remark}
The equivalence between (i) and (iii) holds whenever
$L^{(1)}$ is the maximal restriction of $\widetilde{L}$ on $\ell^1 (V,m)$.
Condition (A) comes into play as it  ensures that $L^{(1)}$ is this maximal restriction.
\end{remark}

%%%%%%%%%%%%%%%%%%%%%%%%%%%%%%%%%%%%%%%%%%%%%%%%%%%%%%%%%%%%%%%%%%%%%%%%%%%%%%%%
%%

\subsection{Recurrence} \label{subParabolicity}
As discussed in the general setting we will restrict our
investigation to irreducible Dirichlet forms. We can characterize
irreducibility in terms of connectedness of the underlying graph
(see \cite{LSV,KLVW} as well).

\begin{lemma}\label{lemma;irreducible}
Let $Q$ be the regular Dirichlet form associated to $(b,c)$. The following
assertions are equivalent:
\begin{itemize}
\item[(i)] $Q$ is irreducible
\item[(ii)] $(b,c)$ is connected
\end{itemize}
\end{lemma}
\begin{proof}
(i) $\Longrightarrow$ (ii): Let $W$ be a connected component of $V$ with respect to $(b,c)$. We show $1_W \cdot u \in D(Q)$ for any $u \in D(Q)$ and
\begin{gather}   Q(u,u) = Q(1_{W} u,1_W u) + Q(1_{W^c} u, 1_{W^c} u) \tag{$\heartsuit$} \label{formula;irreducible}
\end{gather}
holds, where $1_W$ denotes the characteristic function of the set $W$. Since $Q$ is a restriction of $\ow{Q}$ and $W$ is a connected component, the formula \eqref{formula;irreducible} follows as soon as we show $1_W\cdot u \in D(Q)$. This is immediate for $u \in C_c(V)$. Let $u \in D(Q)$ be arbitrary. By regularity there exists a sequence $(u_n)$ in $C_c(V)$ such that $u_n \to u$ with respect to $\aV{\cdot}_Q$. Since \eqref{formula;irreducible} holds for compactly supported functions, we obtain
$\aV{1_W\cdot u_n - 1_W\cdot u_m}_Q \leq \aV{ u_n - u_m}_Q$  and, because $Q$ is closed, we infer $1_W\cdot u \in D(Q)$. From  irreducibility we conclude $W = \emptyset$ or $V\setminus W = \emptyset$ showing the connectedness.

(ii) $\Longrightarrow$ (i): It is not hard to see that \eqref{formula;irreducible} can only be satisfied if $W$ is connected. By the connectedness of the graph we infer $W = V$ or $W = \emptyset$ showing the irreducibility of $Q$.
\end{proof}

With these preparations Theorem \ref{recurrence,regular} reads in the graph situation as follows.

\begin{theorem} \label{parabolicgraph}
Let $Q$ be the regular Dirichlet form
associated to a connected graph $(b,c)$. Then the following assertions are equivalent:

\begin{itemize}
\item[(i)] Q is recurrent.
\item[(ii)] For all $u \in \widetilde{D}$ with $\tilde{L}u \in \ell^1 (V,m)$
the following equality holds
$$ \sum_{x\in V} \widetilde{L}u(x)m(x) = 0.$$

\item[(iii)] For all $u \in \widetilde{D}$ with $\tilde{L}u \in \ell^1 (V,m)$
and all $v \in \ow{D}\cap \ell^\infty$ the following equality holds
$$\ow{Q}(u,v) = \sum_{x\in V} (\ow{L}u)(x)v(x)m(x).$$
\end{itemize}
\end{theorem}

\begin{remark} A result related to (i) $\Leftrightarrow$ (iii) of the previous
theorem can be found in \cite{JP}.
There, a boundary term of the form $Q(u,v) - \as{\ow{L}u,v}$ is defined via a
limiting procedure on
a set of functions which is somewhat different from ours. It is then shown that recurrence is equivalent to this boundary term vanishing.
\end{remark}

A result that is related but somewhat independent of the theory developed in Section~\ref{section;Regular} is the following. It shows that functions of finite energy can be replaced by bounded functions. The proof is given below.

\begin{theorem}\label{parabolicgraphvar}
Let $Q$ be the regular Dirichlet form
associated to a connected graph $(b,0)$.
Then the following assertions are equivalent:

\begin{itemize}
\item[(i)] Q is recurrent.
\item[(ii)] For all $u \in \ell^\infty (V)$ with  $\ow L u \in \ell^1(V,m)$ the following
equality holds
\[
\sum_{x \in V} \ow L u (x) m (x) =0.
\]
\end{itemize}
\end{theorem}

\begin{remark}  The above characterization is a direct analogue to  \cite[Theorem 1.1]{GM}.
\end{remark}

In order to prove the theorem, we establish some notation and
prepare a lemma. For any point $o \in V$,  consider the inner
product
\[ \langle u,v \rangle_o = \widetilde{Q} (u,v) + u(o)v(o) \]
on $V$. If the graph $(b,0)$ is connected the pair $(\widetilde{D}, \langle \cdot, \cdot \rangle_{o})$ is a Hilbert space
and pointwise evaluation of functions is continuous with respect to the
corresponding norm, c.f. \cite[Lemma 3.6]{GHKLW}. Let $\{\Omega_n\}_{n \ge 1}$ be an exhaustion of $V$ with
finite sets
such that each of $\Omega_n$ includes $o \in V$. Furthermore, for $x \in V$ and
finite $G \subseteq V$,  we use the notation
$${\rm cap}(x) = \inf \{Q(u) \mid u \in D(Q), u(x) \geq 1\}$$ and  $${\rm
cap}(x,G) = \inf \{Q(u) \mid u \in D(Q), u(x) \geq 1, u|_{V\setminus G } \equiv
0\}. $$

\begin{lemma} \label{approxcap}
Let $Q$ be the regular Dirichlet form
associated to a connected graph $(b,0)$. Then, the following assertions hold.
\begin{itemize}
\item[(a)] There exists a unique $e \in \overline{D(Q)}^{\|\cdot\|_o}$ such that $\widetilde Q(e) = {\rm{cap}}(o)$ and
unique $e_n \in D(Q)$  such that $Q(e_n) = {\rm{cap}}(o, \Omega_n)$ for
every $n\geq 1$. Furthermore, these functions satisfy $0\leq e_n,e \leq 1$ and $e_n(o) = e(o)=1$.
\item[(b)] The inequality $\ow{L}e_n(x) \geq 0$ holds for every $x \in \Omega_n$, $n\geq 1$.
\item[(c)] ${\rm{cap}}(o, \Omega_n) \to {\rm{cap}}(o)$ and $e_n \to e$
pointwise, as $n \to \infty$.
\item[(d)] If $Q$ is recurrent, then ${\rm cap}(x) = 0$ for any $x \in V.$ 
\end{itemize}
\end{lemma}
\begin{proof}
(a) Let $\|\cdot\|_o$ denote the norm induced by
$\as{\cdot,\cdot}_o$. Then, $e$ and $e_n$ are minimizers of the
functional $u \mapsto \|u\|_o^2 - 1$ on the sets $\{u \in
\overline{D(Q)}^{\aV{\cdot}_o}\mid u(o) \geq 1\}$ and $\{u \in
D(Q)\mid u(o) \geq 1, u|_{V\setminus G }\}$ respectively. The
existence and uniqueness of $e$ and $e_n$ follow from the closedness
and the convexity of these two sets (note that pointwise evaluation
of functions is continuous w.r.t. $\aV{\cdot}_o$) and standard
Hilbert space theory. The furthermore statement follows from the fact that $\| (u\wedge 1)\vee 0\|_o \leq \|u\|_o$ for each $u \in \ow{D}$.

(b) Let $x \in \Omega_n$ be given and $\delta_x$ be the function
which is $\frac{1}{m(x)}$ at $x$ and $0$ elsewhere. For each
$\varepsilon > 0$, we obtain
$$Q(e_n) \leq Q(e_n + \varepsilon \delta_x) = Q(e_n) + \varepsilon^2Q(\delta_x)
+ 2\varepsilon Q(e_n,\delta_x), $$
which together with Lemma \ref{Green} implies
$$\ow{L}e_n(x) = \as{\ow{L}e_n,\delta_x} =  Q(e_n,\delta_x) \geq 0.$$
(c) Obviously, we have ${\rm cap}(o,\Omega_n)\geq {\rm cap}(o)$. We
next show the inequality  $\limsup_{n\to \infty} {\rm cap}
(o,\Omega_n) \leq {\rm cap}(o)$. Choose a sequence of finitely
supported functions $(\varphi_k)$ with
$$\|\varphi_k - e\|_o \to 0, \text{ as } k \to \infty.$$
This convergence implies pointwise convergence. Thus, we infer $\varphi_k(o) \to
1$. Since $(\Omega_n)$ is an exhausting sequence and each $\varphi_k$ has
finite support, we obtain
$$\limsup_{n\to \infty} {\rm cap} (o,\Omega_n) \leq
Q\left(\frac{1}{\varphi_k(o)}\varphi_k\right) =
\frac{1}{\varphi_k(o)^2}Q(\varphi_k) \to \ow{Q}(e) = {\rm cap}(o),
\quad\text{ as } k\to \infty.$$
This gives the desired inequality. It remains to show the statement
on pointwise convergence.
 Using the parallelogram identity
and the convexity of $\{u \in \overline{D(Q)}^{\aV{\cdot}_o}\mid
u(o) \geq 1\} $, we obtain
$$\aV{\frac{e_n - e}{2}}^2_o \leq \frac{1}{2}\aV{e_n}_o^2  +
\frac{1}{2}\aV{e}_o^2 - {\rm cap}(o) - 1 \to 0,\quad\text{ as } n \to \infty. $$
(d) This is a consequence of \cite[Theorem 2.12]{Woe} and the well known fact that the notion of recurrence given in \cite{Woe} coincides with the one for Dirichlet forms used above, c.f. \cite{Sch} for details. This finishes the proof.
\end{proof}

Now, we are in a position to prove the theorem.
\begin{proof}[Proof of Theorem \ref{parabolicgraphvar}]
(i) $\Longrightarrow$ (ii):
Let $e_n$ and $e$ be as in Lemma \ref{approxcap}. Without loss of generality let $0 \le  u (x) \le 1$ for every $x \in V$.
 Set $|\nabla u |^2 (x) := \sum_{y \in V}
     b(x,y) \left( u(x) - u(y)\right)^2$.
Since $e_n$ has finite support, Lemma \ref{Green} implies
$$
 \sum_{x \in V} e_n^2 (x) \ow L u (x) m (x)
=\sum_{x \in V} \ow L  e_n^2 (x) u (x) m (x).
$$
Furthermore, the equation
$$(e_n^2(x) - e_n^2(y)) u(x) = -  u(x) (e_n(x) - e_n(y))^2 + 2 e_n(x) u(x)(e_n(x) - e_n(y)) $$
and the definition of $\ow{L}$ yields
$$\sum_{x \in V} \ow L  e_n^2 (x) u (x) m (x) \\
=-\sum_{x \in V} | \nabla e_n (x)
     |^2 u (x) + 2\sum_{x \in \Omega_{n}}
     e_n (x) \ow L e_n (x) u (x) m (x).$$
By the assumption the form $Q$ is recurrent. Thus, by Lemma~\ref{approxcap}~(d) ${\rm cap}
(o) =0$ and, therefore, $ e\equiv 1$ holds. From the pointwise convergence of the
$e_n$ to $e$ we then obtain $e_n (x) \to 1$ for all $x\in V$. Since the $e_n$ are uniformly bounded,
we infer from Lebesgue's Theorem  that
$$\sum_{x \in V} e_n^2 (x) \ow L u (x) m (x) \to \sum_{x \in V} \ow L u (x) m (x),\text{ as }n \to \infty.$$
 Lemma~\ref{approxcap}~(c) and $0 \leq u \leq 1$ implies
$$\sum_{x \in V} | \nabla e_n (x)
     |^2 u (x) \leq \sum_{x \in V} | \nabla e_n (x)
     |^2 = 2 Q(e_n) \to 0, \text{ as } n \to \infty.$$
Altogether, these considerations yield
$$ \sum_{x \in V} \ow L u (x) m (x) = \lim_{n \to \infty}  2\sum_{x \in \Omega_{n}}
     e_n (x) \ow L e_n (x) u (x) m (x).$$
Since $e_n$ is superharmonic on $\Omega_n$ and $u$ is positive, we obtain
\[\sum_{x \in V} \ow L u (x) m (x) \ge 0.\]
The same argumentation may be repeated with $1-u$ in place of $u$,
and we arrive at the conclusion.

(ii) $\Longrightarrow$ (i): It follows from Theorem
\ref{parabolicgraph} that recurrence is independent of the
underlying measure $m$. Hence, we may assume $m(V) < \infty.$ Now,
assume $Q$ is transient. From Theorem \ref{regular,connection} we
infer $Q^{\rm max} \neq Q$. Theorem~\ref{H neq H0}  and the
characterization of $\cL$ in the graph case imply the existence of a
nontrivial function $u \in \ow{D}\cap\ell^1(V,m)\cap\ell^\infty(V)$
such that $\ow{L}u \in \ell^2(V,m)$ with $\ow{L}u \leq 0$ and
$\ow{L}u \neq 0$. As $m$ is a finite measure $\ell^2(V,m) \subseteq
\ell^1(V,m)$ holds and, therefore, $\ow{L}u \in \ell^1(V,m)$. From
the choice of $u$, we then infer
$$\sum_{x \in V} \ow{L}u(x)m(x) < 0,$$
which shows the claim.\end{proof}

\section{Application to weighted manifolds}\label{section;Manifold}
%%%%%%%%%%%%%%%%%%%%%%%%%%%%%%%%%%%%%%%%%%%%%%%%%%%%%%%%%%%%%%%%%%%%%%%%%%%%%%%%
%%%%%%%%%%%%%%%%
In this section we apply the theory of Section \ref{section;Regular}
to the standard Dirichlet energy on a weighted manifold. This is
exactly the situation that was studied in \cite{GM}. As this setting
is standard we only give a brief introduction and refer the reader
to \cite{Grigoryan.09} for a detailed discussion of the relevant
analysis on manifolds.

Let $(M,g)$ be a connected smooth Riemannian manifold and let $\Phi$
be a strictly positive smooth function on $M$. The triplet $(M,g,
m)$, where $m = \Phi d\mathrm{vol}_{g}$ and $d\mathrm{vol}_{g}$ is
the Riemannian measure, is called a \emph{weighted manifold}. Such a
manifold is a metric measure space which has the same distance and
shape as the underlying Riemannian manifold but its measure is
arbitrary. A weighted manifold carries a second-order elliptic
operator, called the \emph{weighted Laplacian}, defined as
$$
\Delta_\Phi u = \frac{1}{\Phi} {\rm{div}} (\Phi \nabla u),
$$
where the involved operators $\nabla$ and ${\rm div}$ are understood in the distributional sense. The energy form formally corresponding to this operator is given by
$$Q(u,v) = \int_M g(\nabla u, \nabla v) dm.$$
It is well known that $(Q,C_c^\infty(M))$  is closable on $L^2(M,m)$
and its closure, which will be denoted by $(Q,W^{1}_{0}(M,m))$, is a
regular Dirichlet form.

Furthermore, we will need the space  $$W^{1}(M,m) = \{u \in L^2(M,m)
\, | \, |\nabla u| \in L^2(M,m)\},$$ where $|X| =
g(X,X)^\frac{1}{2}.$
\subsection{Functions of finite energy}

We  determine the space of functions of finite energy
$\widetilde{D}(Q)$ and the formal operator $\cL$ associated with
$(Q,W^{1}_{0}(M,m))$.  We will need the  definitions given in
Section \ref{subsection;FFE}.

\begin{prop}[Reflected Dirichlet space and Beppol-Levi functions]
The reflected Dirichlet space of $(Q,W^{1}_{0}(M,m))$ is given by the space
of Beppo-Levi functions
$${\rm BL}(M,m) = \{u \in L^2_{\rm loc}(M) \, | \, |\nabla u|  \in
L^2(M,m)\}.$$
On this space the extended form $\ow{Q}$ acts as
$$\ow{Q}(u,v) = \int_M g(\nabla u,\nabla v) \, dm.$$
In particular, the Dirichlet form $Q^{\rm max}$ is the restriction of $\ow{Q}$ to the space $W^1(M,m)$. The domain of the operator $\mathcal{L}$ is given by
$$D(\mathcal{L}) = \{ u \in {\rm BL}(M,m)\, | \, \Delta_\Phi u \in L^1_{\rm
loc}(M)\}$$
on which it acts as $\mathcal{L}u = - \Delta_\Phi u.$
\end{prop}

\begin{proof} ${\rm BL} (M,m) \subseteq \ow{D}(Q)$: Let $u \in {\rm BL}(M,m)$ be bounded
and let a smooth function with compact support $\varphi$ be given.
It is well known that $u \cdot \varphi \in W_{0}^{1}(M,m)$  (see
e.g. \cite[Corollary~5.6]{Grigoryan.09}), which gives that  $u$
agrees locally with a function from $ W^{1}_{0}(M,m)$.  Now, the
remaining statements are  a rather  direct consequence of the
definitions and the Markov property of the weighted energy on ${\rm
BL}(M,m)$ (compare proof of Theorem \ref{Fatou} for a similar
reasoning).

$\ow{D}(Q) \subseteq {\rm BL}(M,m)$: Let $u \in \ow{D}(Q)$ and a
relatively compact, open subset $G\subseteq M$ be given. As the
positive and negative part of $u$ belong to $\ow{D}(Q)$, we may
assume $u \geq 0$. Recall that we set $u^{(n)} = (u \wedge n)\vee
(-n)$. By definition $u^{(n)} \in {\rm BL}(M,m)$ for each $n\in \NN$
and $\sup_n Q(u^{(n)}) <\infty$. Furthermore, as $G$ is relatively
compact, a Poincar\'{e}-type inequality holds on $G$, that is there
is a constant $C > 0$ (depending on $G$) such that
$$\int_G |v - \overline{v}_G|^2\, dm \leq C \int_M g(\nabla v, \nabla v)\, dm$$
for each $v \in {\rm BL}(M,m).$ Here, $\overline{v}_G = m(G)^{-1} \int_G v \, dm$. This  rough Poincar\'e
inequality holds since on compact sets the Ricci curvature is bounded from
below and $\Phi^{-1}$ is locally bounded. We define $f := |u -
\overline{u}_G|^2$ and $f_n = |u^{(n)} - \overline{u^{(n)}}_G|^2.$ The function
$u$ is positive and almost surely finite. Hence, $f$ is well defined.
Furthermore, $f$ is almost surely finite if and only if $\overline{u}_G <
\infty$. We deduce from monotone convergence and the positivity of $u$ that
$\overline{u}_G = \lim_n \overline{u^{(n)}}_G$. Using  this observation and
applying the Poincar\'{e} inequality, we obtain with the help of Fatou's lemma
$$\int_G f\, dm \leq \liminf_{n\to \infty} \int_G f_n\, dm \leq C
\liminf_{n\to \infty}Q(\nabla u^{(n)}) < \infty.$$
This implies $\overline{u}_{G} < \infty$ and $u \in L^2(G)$. As $G$ was
arbitrary we obtain $u \in L^2_{\rm loc}(M)$. Furthermore, $\nabla u^{(n)}$ is
a bounded sequence in the Hilbert space $\vec{L}^2(M,m)$, i.e., the space of $L^2$ vector fields. Thus, it possesses a weakly convergent subsequence with weak limit $V \in \vec{L}^2(M,m)$. Choose a smooth
vector field $X$ with compact support. We obtain
\begin{align*}
\int_M g(V,X)\, dm  &= \lim_{k\to \infty} \int_M g(\nabla u^{(n_k)},X)\,dm \\
&= -\lim_{k\to \infty}\int_M u^{(n_k)} \frac{1}{\Phi}{\rm div}(\Phi X)\,dm\\
 &= -\int_M u \frac{1}{\Phi}{\rm div}(\Phi X)\,dm.
\end{align*}
 As $X$ was arbitrary, we obtain $\nabla u = V \in \vec{L}^2(M,m)$ which show the claim.

The statement about the action of $Q$ follows from the definition of the
extension of $Q$ to $\ow{D}(Q)$. Furthermore, the statement on $\cL$ follows
from the definition of $\Delta_\Phi$ via distributions.
\end{proof}

\begin{remark}
For Brownian motion on an open subset of $\mathbb{R}^n$ the
statement about the reflected Dirichlet space is one of the examples
of Section 6.5 in \cite{CF}. For general Riemannian manifolds the
statement  seems to be new.
\end{remark}

\subsection{Extensions of Dirichlet forms}

Theorem \ref{H neq H0} now reads in the manifold setting.

\begin{theorem}\label{thm;MF;H neq H0}
The following assertions are equivalent.
\begin{itemize}
\item[(i)] $W^1(M,m) \neq W^1_0(M,m)$.
\item[(ii)] There exists a nontrivial $u \in W^1(M,m)\cap L^1(M,m)\cap L^{\infty}(M,m)$
such that $\Delta_\Phi u \in L^2(M,m)$ with $u\geq 0$, $ \Delta_\phi u \geq 0$ and $ \Delta_\Phi u \neq
0$.
\end{itemize}
\end{theorem}
Next, we come to the application of Theorem~\ref{thm;US1}.

\begin{theorem} \label{thm;MF;US1}
Let $L$ be the self-adjoint operator associated with
$(Q,W^1_0(M,M))$. Then, the following assertions are equivalent.
\begin{itemize}
\item[(i)] $W^1(M,m) = W^1_0(M,m)$.
\item[(ii)] $D(L) = \{u \in  W^1(M,m )\mid \Delta_\Phi u \in L^2(M,m)\}$.
\item[(iii)] For all $u \in  W^1(M,m )$ such that $\Delta_\Phi \in L^2(M,m)$
and all $v \in W^1(M,m)$ we have
$$\ow{Q}(u,v) = - \int_M \Delta_\Phi u v \,dm.$$
\end{itemize}
\end{theorem}

\subsection{Stochastic completeness}

The following is the application of Theorem \ref{SC-Reg} to the manifold setting.

\begin{theorem} The following assertions are equivalent:
\begin{itemize}
\item[(i)] $(Q,W^{1}_0(M,m))$ is stochastically complete.
\item[(ii)] For all $u \in L^1(M,m) \cap L^2(M,m)$ such that $|\nabla u| \in
L^2(M,m)$ and $\Delta_\Phi \in L^1(M,m) \cap L^2(M,m)$ the following equality holds
$$\int_M \Delta_\Phi u \, dm = 0.$$
\end{itemize}
\end{theorem}
\begin{remark}
The previous theorem is basically the same as of
\cite[Theorem~1.2]{GM} after one realizes  that stochastic
completeness implies $W^{1}(M,m) = W^{1}_0(M,m)$ (which follows e.g.
from the implication (ii)$\Longrightarrow$ (iii) in  Theorem
\ref{regular,connection}).
\end{remark}

\subsection{Recurrence}

In the manifold setting Theorem \ref{recurrence,regular} becomes:

\begin{theorem}
The following assertions are equivalent.
\begin{itemize}
\item[(i)] $(Q,W^{1}_{0}(M,m))$ is recurrent.
\item[(ii)] For each $u \in {\rm BL}(M,m)$ such that $\Delta_\Phi u \in
L^1(M,m)$ the following equality holds
$$\int_M \Delta_\Phi u\, dm = 0.$$
\item[(iii)] For each $u \in {\rm BL}(M,m)$ such that $\Delta_\Phi u \in
L^1(M,m)$ and each $v \in {\rm BL}(M,m) \cap L^\infty(M,m)$ the following equality holds
$$\ow{Q}(u,v) = -\int_M \Delta_\Phi u v \,dm.$$
\end{itemize}
\end{theorem}

\begin{remark} The previous theorem is an analogue to  \cite[Theorem~1.1]{GM}. There the equivalence of (i) and (ii) is also proven but, in contrast to
our result, with the Beppo-Levi functions replaced by $L^\infty$ functions. However, the advantage of using Beppo-Levi functions is that the very general
Green type formula (iii) holds for them in the recurrent case.
\end{remark}
%%%%%%%%%%%%%%%%%%%%%%%%%%%%%%%%%%%%%%%%%%%%%%%%%%%%%%%%%%%%%%%%%%%%%%%%%%%%%%%%
%%%%%%%%%%%%%%%%
\section{Application to metric graphs}\label{section;MG}

In this section we consider metric graphs and discuss applications of the
abstract results of the previous sections. Metric graphs are in some sense a hybrid model between manifolds and
discrete graphs and fit into the framework of regular
Dirichlet forms. Most of the material presented here is
based on the thesis \cite{Hae} of one of the authors.

The basic idea of a metric graph is to view edges as intervals which are glued
together according a graph structure.

Let $l$ be a locally finite graph over a discrete countable vertex set $V$, i.e., $l:V\times V \to [0,\infty)$ is symmetric, has zero diagonal and $l(x,\cdot)$ vanishes for all but finitely many vertices. We define the set of combinatorial edges $E$ to be the equivalence classes of
$\{(x,y)\subseteq V\mid l(x,y)>0\}$ under the equivalence relation that relates $(x,y)$ and $(y,x)$ for all $x,y\in V$. By symmetry of $l$ the map $l$ is well defined on $E$.

For $e\in E$, we define the \emph{continuum edge} $\CX_e = (0,l(e)) \times
\{e\}$ and  the \emph{metric graph} to be the set
\[\CX_\Gamma = V \cup \bigcup_{e\in E} \CX_e .\]

Next, we equip $\CX_{\Gamma}$ with a topology defined in terms of a certain subspace of continuous functions. We define an orientation on the combinatorial edges which is a map $E\to\{(x,y)\subseteq V\mid l(x,y)>0\}$, $e\mapsto(\partial^{+}(e),\partial^{-}(e))$. We call $\partial^+(e)$ the \emph{initial vertex}  and the  $\partial^-(e)$ \emph{terminal vertex} of $e$. If $x$ is the initial or terminal vertex of an edge $e$ we say $x$ and $e$ are \emph{adjacent} and we write $ x\sim e$.

For a vertex $x$ adjacent to an edge $e$ and variables $t\in(0,l(e))$, we interpret $t\to x$ as $t\to 0$ if $x=\partial^{+}(e)$ and as $t\to l(e)$ if $x=\partial^{-}(e)$.

For a function $f:\CX_{\Gamma}\to\R$,  we denote by $f_{e}$ the restriction of $f$ to $\CX_{e}$ which is essentially a function on $(0,l(e))$. The continuous functions $C(\CX_{\Gamma})$ are the functions $f:\CX_{\Gamma}\to\R$ such that $f_{e}$ is continuous for all $e\in E$ and for all $x \in V$ we have
\begin{align*}
    f(x)=\lim_{t\to x}f_{e}(t)
\end{align*}
for all $e\in E$ adjacent to $x$. The space $C(\CX_{\Gamma})$ gives rise to a topology on $\CX_{\Gamma}$ such that $\CX_{\Gamma}$ becomes a locally compact Hausdorff space. We denote by $C_{c}(\CX_{\Gamma})$ the subspace of continuous functions of compact support. For a function $f$ such that $f_{e}$ are weakly differentiable for all $e\in E$, we write $f'=(f_{e}')_{e\in E}$ and $f''=(f_{e}'')_{e\in E}$ similarly.

We introduce  Lebesgue spaces $L^{p}(\CX_{\Gamma})=\bigoplus_{e\in E}L^{p}(0,l(e))$, % and the Sobolev spaces $W^{p,q}(\CX_{\Gamma})=\bigoplus_{e\in E} W^{p,q}(0,l(e))$,
$p\in\{1,2\}$, where we neglect the vertices since points have Lebesgue measure zero. We denote the space of functions locally in $L^{p}(\CX_{\Gamma})$, $p\in\{1,2\}$, by $L^{p}_{\rm loc}(\CX_{\Gamma})$. This space is no direct sum since locally here means with respect to the topology of $\CX_{\Gamma}$. For a more detailed description of the set up, we refer to \cite[Chapter~1]{Hae}.

Now, let $b$ be another locally finite graph over $V$ such as in Section~\ref{section;Graph}. Note that the combinatorial structure of $l$ and $b$ is completely independent. We introduce the quadratic form
\[\oh{Q}(u)= \sum_{e\in E} \int\limits_0^{l(e)} |u_e'(t)|^2 dt +
\frac{1}{2} \sum_{x,y\in V} b(x,y)(u(x)-u(y))^2\]
on the space which will turn out to be the space of functions of finite energy
\[\widetilde{D} = \{u\in C(\CX_\Gamma)\mid u'\in L^{2}(\CX_{\Gamma}), \,\oh{Q} (u) <\infty\}. \]
On this space $\oh{Q}$ has the Markov property, i.e., for each $u \in \ow{D}$ and each normal contraction $C: \RR \to \RR$ we have that $C \circ u \in \ow{D}$ and $\oh{Q}(C\circ u) \leq \oh{Q}(u)$. We define $Q$ to be the restriction of $\oh Q$ to
\begin{align*}
    D(Q) = \overline{\widetilde{D}\cap   C_c(\CX_\Gamma)}^{\|\cdot\|_{\oh Q}},
\end{align*}
where $\|\cdot\|_{\oh Q}^2 = \|\cdot\|^2 + \oh{Q}(\cdot)$. It can be checked  that $Q$ is a regular Dirichlet form, for details see \cite[Chapter~1, Section~3]{Hae}.

\subsection{Functions of finite energy}
We will deal with the  functions of finite energy $\widetilde D(Q)$
arising from a regular Dirichlet form $Q$ as introduced in
Section~\ref{subsection;FFE}.

\begin{theorem}\label{thm;MG;FinEnergy} Let $Q$ be the regular Dirichlet form defined on a metric graph as above. Then,
 $\widetilde{D}=\ow D(Q)$ and $\ow{Q}$ on $\ow{D}(Q)$ is given by $\oh{Q}$.
\end{theorem}

In order to prove the preceding theorem we need to characterize
convergence in $\ow D$ which turns out to be a Hilbert space with
respect to the norm $\|\cdot\|_o := (|u(o)|^2 +
\oh{Q}(u))^{1/2}$ for arbitrary $o\in \CX_\Gamma$ whenever the combinatorial graph $l+b$ over $V$ is connected.

\begin{lemma}
Let the graph $l+b$ be connected. The space $\widetilde{D}$ equipped with $\|\cdot\|_{o}$ for $o\in
\CX_\Gamma$ is a Hilbert space. A sequence $(u_n)$ in $
\widetilde{D}$ converges to $u\in \widetilde{D}$ with
respect to $\|\cdot\|_o$ if and only if it converges pointwise to
$u$ and $\limsup_n \oh{Q} (u_n) \leq \oh{Q}(u)$.
\end{lemma}
\begin{proof}
Note that convergence with respect to the $\|\cdot\|_o$ implies pointwise convergence. This is due to a one-dimensional Sobolev embedding which can be deduced as follows. For an arbitrary $x \in \CX_\Gamma,$ let $\gamma$ be a path connecting $x$ and $o$ (which consists of a mix of continuous paths along edges with respect to $l$ and combinatorial paths along edges of $b$). Then, a combination of the fundamental theorem of calculus along the continuous parts of $\gamma$ and a summation along the combinatorial parts of $\gamma$ and Cauchy-Schwarz inequality lead to 
$$|u(x)| \leq C(\gamma) \oh{Q}(u)^{1/2},$$
for arbirtray $u \in \ow{D}$, with a constant $C(\gamma)$ independent of $u$. For more details we refer the reader to \cite[Chapter~1, Section~2]{Hae}.

Now, the pointwise limit has finite energy by  standard Fatou type arguments.
Hence, $\ow D$ is a Hilbert space with norm $\|\cdot\|_{o}$.\\
Next, consider a sequence  $(u_n)$ converging to  $ u$ with respect to
$\|\cdot\|_o$. As mentioned already, this implies   pointwise
convergence of $u_n$ to $u$  and, clearly, $\oh{Q}(u_n) \to
\oh{Q}(u)$ holds as well. For the other direction let $u_n
\in \widetilde{D}$, $n\ge0$, be a sequence as stated. In Hilbert
spaces bounded sets are weakly compact. Since $(u_n)$ is a bounded
sequence, there is a weakly convergent subsequence. This weak limit
$u$ has to agree with the pointwise limit and we have $u_n \to u$
weakly. Finally, we arrive at
\begin{eqnarray*}
0 &\leq & \oh{Q} (u-u_n) + (u(o) - u_n(o))^2 \leq  \oh{Q}(u) + u(o)^2 + \oh{Q}(u_n) + u_n(o)^2 - 2\langle
u,u_n\rangle_o
\end{eqnarray*}
which yields $u_n \to u$ in $\widetilde{D}$.
\end{proof}

\begin{proof}[Proof of Theorem~\ref{thm;MG;FinEnergy}]
We will only prove the statement in the case when the graph $l+b$ is connected. The general case follows by considering connected components. 

Recall that $\{ u\in D(Q)_{\rm loc} \cap L^\infty(X,m) \mid Q(u) < \infty\} = \ow{D}\cap L^\infty(X,m)$ and that the extension of $Q$ to $D(Q)_{\rm loc}$ and $\oh{Q}$ agree on this set. Let $u \in \ow{D}$ be given. As discussed, we obtain $u^{(n)}= (u \wedge n)\vee (-n) \in D(Q)_{\rm loc}$ and by the Markov property of $\oh Q$ we have $Q(u^{(n)}) = \oh Q(u^{(n)}) \leq \oh Q(u)$ for each $n$. Therefore, $u\in \ow D(Q)$ holds.

On the other hand, for $u\in \ow D(Q)$, we conclude $u^{(n)} \in \ow{D}$ and $\oh{Q}(u^{(n)}) = \ow{Q}(u^{(n)}) \leq \ow{Q}(u)$. Therefore, $(u^{(n)})$ is a bounded sequence in the Hilbert space $(\oh{Q},\ow{D})$. Thus, it has a weakly convergent subsequence. By the pointwise convergence of $u^{(n)}$ towards $u$ this limit must coincide with $u,$ showing $u \in \ow{D}$. Since $\oh{Q}$ has the Markov property, we obtain $\oh{Q}(u^{(n)}) \leq \oh{Q}(u^{(n)})$. Now, the previous lemma implies
$$\oh{Q}(u) = \lim_{n\to \infty} \oh{Q}((u \wedge n)\vee (-n)) =  \lim_{n\to \infty} Q((u \wedge n)\vee (-n)) = \widetilde{Q}(u),$$
where the last equality follows from the definition of $\ow{Q}.$ This finishes the proof.
\end{proof}

We now turn to the associated operators. We denote by $\ow F$ the
space from Section~\ref{section;Graph} for the graph $b$ over $V$.
Similarly, we let  $\ow L$ be the generalized Laplacian from
Section~\ref{section;Graph} for the graph $b$ and the counting
measure $m\equiv 1$.

For $u$ such that $u_{e}''\in L^{2}(0,l(e))$, $e\in E$, the derivatives $u_{e}'(\partial^{+}(e))$ and $u_{e}'(\partial^{-}(e))$ exist  for all $e\in E$ and we define normal derivative in a vertex
$x\in V$ by
\[\partial_n u (x)= \sum_{e\sim x} \sum_{\partial^+(e)=x} u_e'(0) -
\sum_{\partial^-(e)=x} u_e'(l(e)).\]
We say $u\in W_{\mathrm{loc}}^{1,2}(\CX_{\Gamma})$ satisfies the Kirchoff conditions
if $u\in C(\CX_{\Gamma})$, $u\vert_{V}\in{\ow F}$ and
\begin{align}\label{eq;KC}\tag{KC}
    \partial_{n}u(x)= \ow Lu(x),\qquad x\in V.
\end{align}
%With this notation we introduce an analogue of the operator $\ow L$ from Section~\ref{section;Graph} which we denote by $\ow \Delta$.
%Let $$ D(\ow \Delta)=\{u\in W^{1,2}(\CX_{\Gamma})\mid u\mbox { satisfies (KC)}\}$$ and
%\begin{align*} (\ow   \Delta u)_{e}=-u_{e}'',\qquad e\in E. \end{align*}

We will need the  operator $\mathcal{L}$ defined at the end of
Section~\ref{subsection;FFE}.

\begin{theorem}
The operator $\mathcal{L}$ acts as
\begin{align*} (\mathcal{L}u)_{e}=-u_{e}'',\qquad e\in E, \end{align*}
on the domain
\begin{align*}
D(\mathcal{L})=\{u\in C( \CX_\Gamma) \mid u'\in L^{2}(\CX_\Gamma), u''\in L^{1}_{\mathrm{loc}}(\CX_\Gamma), u\mbox{ satisfies (KC)}\}.
\end{align*}
\end{theorem}
\begin{proof} The domain $D(\mathcal{L})$ is given as
$$\{u \in \widetilde{D}(Q)\, | \, \text{ there exists } f \in
L^1_{\rm loc}(\CX_{\Gamma}) \text{ s.t. } \ow{Q}(u,v) = \as{f,v}
\text{for all } v \in D(Q) \cap C_c(\CX_{\Gamma})\}.$$ For $u$ in
this domain, we get  from Theorem~\ref{thm;MG;FinEnergy}  $u\in \ow
D(Q)=\ow D$ which implies $u'\in L^{2}( \CX_\Gamma)$. Furthermore,
by testing with functions $v$ supported within the edges, we get
$f=u''\in L^{1}_{\mathrm{loc}}(\CX_{\Gamma})$. Finally, we see using
partial integration and Lemma~\ref{Green},
\begin{align*}
\langle - u'', v\rangle=\widetilde{Q}(u,v) =&\langle
-\widetilde\partial_{n} u, v\rangle - \langle u'', v\rangle+ \langle  \widetilde{L} u, v \rangle.
\end{align*}
Hence, $u$ satisfies (KC). The other inclusion follows by similar considerations.
\end{proof}

\begin{remark} Let us note that the values of $b$ do not play a role
in the action of $\mathcal{L}$. They only enter when it comes to
domains.
\end{remark}

\subsection{Extensions of Dirichlet forms}

By Theorem~\ref{thm;MG;FinEnergy} the form $Q^{\max}$ introduced in Section~\ref{subsection;Silverstein} is a restriction of $\ow Q$ to
\begin{align*}
    D(Q^{\max})=\ow D\cap L^{2}(\CX_{\Gamma}).
\end{align*}

The following is an application of Theorem~\ref{H neq H0}.

\begin{theorem}\label{thm;MG;H neq H0}
The following assertions are
equivalent.\begin{itemize}
\item[(i)] $Q \neq Q^{\max}$.
\item[(ii)] There exists a nontrivial $u \in D(\mathcal{L})\cap L^1(\CX_{\Gamma})\cap L^2(\CX_{\Gamma})$
such that $ u'' \in L^2(\CX_{\Gamma})$ with $u\geq 0$, $ u ''\leq 0$ and $ u'' \neq
0$.
\end{itemize}

\end{theorem}
The following is an application of Theorem~\ref{thm;US1}.
\begin{theorem} \label{thm;MG,US1}
The following assertions are equivalent.
\begin{itemize}
\item[(i)] $Q = Q^{\max}$.
\item[(ii)] $D(L) = \{u \in D(\cL )\cap L^2(\CX_{\Gamma})\, | \,  u'' \in L^2(\CX_{\Gamma})\}$.
\item[(iii)] For all $u \in D(\cL)\cap L^2(\CX_{\Gamma})$ such that $ u'' \in L^2(\CX_{\Gamma})$
and all $v \in \ow{D}(Q)\cap L^2(X,m)$
$$\ow{Q}(u,v) = \sum_{e\in E}\int_0^{l(e)} u''_{e}(t)v(t) \, dt.$$
\end{itemize}
\end{theorem}

\subsection{Stochastic completeness}
Next, we come to the application of Theorem~\ref{SC-Reg}.
\begin{theorem} \label{thm;SC-MG}
The following assertions are
equivalent.
\begin{itemize}
\item[(i)] $Q$ is stochastically complete.
\item[(ii)] For all $u \in D(\mathcal{L})\cap L^1(\CX_{\Gamma}) \cap L^2(\CX_{\Gamma})$  and $ u''
\in L^1(\CX_{\Gamma}) \cap L^2(\CX_{\Gamma})$ we have
$$\sum_{e\in E}\int_0^{l(e)} u''_{e}(t) \, dt = 0.$$
\end{itemize}
\end{theorem}

\subsection{Recurrence}

In Section~\ref{section;Graph} we discussed what it means that a graph is connected. The form $Q$ has two underlying graphs. The graph $l$ giving rise to $\CX_{\Gamma}$ and the graph $b$ giving rise to the jumping part of $Q$.

The Dirichlet form $Q$ is irreducible if the graph $l+b$ is connected. This follows by the same argument as  in Lemma~\ref{lemma;irreducible}. With this at hand we can apply Theorem~\ref{recurrence,regular}.

\begin{theorem} \label{thm;recurrence,MG}
Assume the graph $l+b$ is connected. Then the following
assertions are equivalent.
\begin{itemize}
\item[(i)] $Q$ is recurrent.
\item[(ii)]  For all $u \in  D(\mathcal{L})$  and $u''\in L^1(\CX_{\Gamma}) $ we have
$$\sum_{e\in E}\int_0^{l(e)} u''_{e}(t) \, dt = 0.$$

\item[(iii)]  For all $u \in  D(\mathcal{L})$  and $u''\in L^1(\CX_{\Gamma})$ and for all $v \in \widetilde{D}\cap
L^\infty(\CX_{\Gamma})$ we have
$$\ow{Q}(u,v) = -\sum_{e\in E}\int_0^{l(e)} u''_{e}(t)  v_e(t) \,dt.$$
\end{itemize}
\end{theorem}

%%%%%%%%%%%%%%%%%%%%%%%%%%%%%%%%%%%%
   \appendix

\section{Construction  of the sequence $(e_n)$ and a lemma on sequences}
\label{section;A} The subsequent two results are certainly known in
one way or other. We include a proof in this appendix in order to
keep the paper self-contained and as they may be useful for further
references as well.

%In this section we provide a proof of the following proposition
%which was used in the proof of Lemma~\ref{SC-two} on the way to
%prove Theorem~\ref{thm;SC}.

   \begin{prop} \label{choice-e-n}  Let $Q$ be a Dirichlet form on a
$\sigma$-finite space $(X,m)$. Then, the following holds:
\begin{itemize}
  \item [(a)] There exists a  sequence $(e_n)$ in $D(Q)$ with $0\leq e_n \leq 1$ and
$e_n \to 1$ $m$-a.e.
  \item [(b)] If $Q$ is  regular, the sequence  $(e_n)$ from (a)  can be
chosen in $D(Q)\cap C_c (X)$.
  \item [(c)] If $Q$ is regular and recurrent, then  $(e_n)$ from (b) can be chosen
to satisfy  $e_n \to 1$ in the sense of $Q_e$, i.e., $Q(e_n)\to
0$.
\end{itemize}
\end{prop}
\begin{proof} (a) By our assumptions there exists an increasing sequence of
sets of finite measure $(B_k)$ such that $X = \bigcup_k B_k$.
 Because $D(Q)$ is dense in $L^2(X,m)$, we can choose $f_n \in D(Q)$ satisfying
 $$ \|f_n - \chi _{B_n}\|_{2} \to 0, \text{ as }n \to \infty.$$
 Let $e_n = (0 \vee f_n)\wedge 1$. Since $Q$ is a Dirichlet form, we infer  $e_n \in D(Q).$ Furthermore by
construction we see $0 \leq e_n \leq 1$ and $$ \|e_n - \chi _{B_n}\|_{2} \leq
\|f_n - \chi _{B_n}\|_{2}.  $$
  We want to show that $(e_n)$ possesses a subsequence converging to $1$
$m-$almost surely. For  $k,n \in \mathbb{N}$ and $\delta > 0$ let
$$ A_{k,n,\delta} = \{x \in B_k: |e_n(x)-1| \geq \delta \}.$$
By the Markov-inequality we observe for $n \geq k$
$$m(A_{k,n,\delta}) \leq \delta^{-2} \| (1-e_n)\chi_{B_k} \|_2^2  \leq
\delta^{-2}\| \chi_{B_n} - e_n \|_2^2.$$
This allows us to choose a subsequence $e_{n_l}$, such that for any $k$
$$ \sum_{l=1}^{\infty} m(A_{k,n_l,l^{-1}}) < \infty.$$
Let $ N = \bigcup_k \bigcap_{j \geq 1} \bigcup_{l\geq j}A_{k,n_l,l^{-1}}$ and
$x \in X\setminus N$. It is easily verified that $e_{n_l}(x) \to 1$ as   $l \to
\infty$. To prove (a), it remains to show $m(N) = 0$, which can be checked
directly by computing
\begin{align*}
m(N)  \leq \sum_k \lim_{j \to \infty} m \Big(\bigcup_{l \geq j}
A_{k,n_l,l^{-1}}\Big)  \leq \sum_k \limsup_{j \to \infty} \sum_{n\geq j} m(A_{k,n_l,l^{-1}}) =0
\end{align*}

(b) Because of the regularity of $Q$, we know that $C_c(X) \cap D(Q)$ is dense
in $D(Q)$ (hence, in $L^2(X,m)$) with respect to $L^2(X,m)$ convergence. Thus, in the proof of (a) we can replace $f_n \in D(Q)$ by $g_n \in  C_c(X) \cap D(Q)$ to obtain (b).

(c) By recurrence there exists a sequence $h_n \in D(Q)$ such that $0 \leq h_n
\leq 1$, $h_n \to 1$ pointwise m-almost surely and
$$ \lim_{n \to \infty} Q(h_n)= 0. $$
By  regularity of $Q$ we can choose $\tilde{e}_n \in C_c(X) \cap D(Q)$
satisfying
$$ \|\tilde{e}_n-h_n \|_Q \to 0 \text{ as } n\to \infty.$$
Let $e_n = (0 \vee h_n)\wedge 1$ such that $0\leq e_n \leq 1$. We will show,
that $e_n$ has a subsequence converging to $1$ $m$-almost everywhere and
$$\lim_{n\to \infty}Q(e_n) = 0.$$
Let $A_{k,n,\delta}$  be sets defined as in the proof of (a). The first
assertion follows as above, using
\begin{align*}
m(A_{k,n,\delta}) &\leq \delta ^{-2} \|(e_n - 1)\chi_{B_k}\|_2^2 \\
&\leq \delta
^{-2} \|(\tilde{e}_n - 1)\chi_{B_k}\|_2^2\\
&\leq \delta ^{-2} \left[ \|(\tilde{e}_n - h_n)\chi_{B_k}\|_2 + \|(1-
h_n)\chi_{B_k}\|_2 \right]^2 \\
&\leq  \delta ^{-2} \left[ \|(\tilde{e}_n - h_n)\|_2 + \|(1- h_n)\chi_{B_k}\|_2
\right]^2.
\end{align*}
The second statement can be deduced by
\begin{align*}
Q(e_n)^ {1/2} \leq Q(\tilde{e}_n)^ {1/2}  \leq  Q(\tilde{e}_n-h_n)^ {1/2} +
Q(h_n)^ {1/2}  \leq \|\tilde{e}_n-h_n\|_Q + Q(h_n)^ {1/2}.
\end{align*}
This finishes the proof.
\end{proof}

\begin{lemma} \label{sequences} Let $(a_{n,m})_{n,m \in \mathbb{N}}$ be a
sequence of real numbers satisfying $a_{n+1,m} \geq a_{n,m}$ for each $n,m \in
\mathbb{N}$. Then,
$$\liminf_{n\to \infty}\liminf_{m\to \infty} a_{n,m} \leq \liminf_{m\to
\infty}\liminf_{n\to \infty} a_{n,m}.$$
\end{lemma}
\begin{proof} Suppose $\liminf_{m\to \infty}\liminf_{n\to \infty} a_{n,m} <
\infty$. Let $\varepsilon > 0$ be arbitrary. Choose an increasing sequence of
indices $(m_l)$ such that
$$\liminf_{n\to \infty} a_{n,m_l} \leq \liminf_{m \to \infty} \liminf_{n\to
\infty} a_{n,m} + \varepsilon$$
for each $l \geq 1$. This and the monotonicity in $n$ imply
$$a_{n,m_l} \leq \liminf_{m \to \infty} \liminf_{n\to \infty} a_{n,m} +\varepsilon.$$
As $(a_{n,m_l})_{l \geq 1}$ is a particular subsequence of $(a_{n,m})_{m \geq1}$, we infer
$$\liminf_{m\to \infty} a_{n,m} \leq \lim_{l \to \infty} \liminf_{n\to \infty}
a_{n,m_l} + \varepsilon$$
which proves the claim.
\end{proof}

\end{document}